\documentclass{amsart}

\usepackage{graphicx,color}
\usepackage{amssymb}
\usepackage{amscd}
\usepackage{longtable}
\usepackage{pb-diagram}
\usepackage[all]{xy}

\newtheorem{theorem}{Theorem}[section]
\newtheorem{proposition}[theorem]{Proposition}
\newtheorem{lemma}[theorem]{Lemma}
\newtheorem{question}[theorem]{Question}
\newtheorem{problem}[theorem]{Problem}

\newtheorem{corollary}[theorem]{Corollary}
\newtheorem*{claim*}{Claim}

\theoremstyle{definition}

\newtheorem{example}[theorem]{Example}

\newcommand{\N}{\Bbb N}
\newcommand{\F}{\Bbb F}

\newcommand{\G}{G(K)}
\newcommand{\C}{\Bbb C}
\newcommand{\Z}{\Bbb Z}
\newcommand{\Q}{\Bbb Q}
\newcommand{\D}{\Delta}

\newcommand{\la}{\langle}
\newcommand{\ra}{\rangle}

\newcommand{\ord}{\mathcal{O}}

\newcommand{\MN}{\mathcal{MN}}
\newcommand{\pr}{\mathrm{pr}}

\theoremstyle{remark}
\newtheorem{remark}[theorem]{Remark}

\numberwithin{equation}{section}

\allowdisplaybreaks[4]

\begin{document}
\title{Twisted Alexander vanishing order of knots II} 

\author{Katsumi Ishikawa, Takayuki Morifuji, and Masaaki Suzuki}



\begin{abstract}
In our previous work, we introduced the notion of the twisted Alexander vanishing order of knots, defined as the order of the smallest finite group for which the corresponding twisted Alexander polynomial vanishes. In this paper, we explore several properties of this invariant in detail and present a list of twisted Alexander vanishing groups of order less than $201$.
\end{abstract}

\address{Research Institute for Mathematical Sciences, Kyoto University, Kyoto 606-8502, Japan}
\email{katsumi@kurims.kyoto-u.ac.jp}

\address{Department of Mathematics, Hiyoshi Campus, Keio University, Yokohama 223-8521, Japan}
\email{morifuji@keio.jp}

\address{Department of Frontier Media Science, 
Meiji University, 4-21-1 Nakano, Nakano-ku, Tokyo 
164-8525, Japan}
\email{mackysuzuki@meiji.ac.jp}

\maketitle

\section{Introduction}

The twisted Alexander polynomial is introduced by Lin \cite{Lin01-1} for knots in the $3$-sphere $S^3$, and by Wada \cite{Wada94-1} for finitely presentable groups. It is a generalization of the classical Alexander polynomial, and is defined for the pair of a group (e.g., a knot group) and its representation. The theory of twisted Alexander polynomials has developed over the last twenty-five years. For example, it is known that fibered $3$-manifolds are detected by twisted Alexander polynomials associated with finite groups (see \cite{FV11-1}). See \cite{FV10-1}, \cite{Morifuji15-1} for other applications to knot theory and low-dimensional topology. 

We call a finite group $G$ a \textit{twisted Alexander vanishing (TAV) group of a knot $K$} in $S^3$ if there exists an epimorphism $f$ of the knot group $G(K)$ onto $G$ such that the twisted Alexander polynomial associated with  the composition of $f$ and the regular representation of $G$ is zero (see \cite{IMS23-1}). Furthermore, we call $K$ a \textit{fibered knot} if the complement of $K$ in $S^3$ admits a structure of a surface bundle over the circle such that the closures of the fiberes are Seifert surfaces. Then, a vanishing theorem for non-fibered knots due to Friedl and Vidussi is described as follows: 

\begin{theorem}[{\cite[Theorem~1.2]{FV13-1}}]\label{thm:FV}
Every non-fibered knot admits a TAV group. 
\end{theorem}

Let us recall the notion of the \textit{twisted Alexander vanishing order} 
$\ord(K)$ of a non-fibered knot $K$ (see \cite{IMS23-1} for details). 
We define $\ord(K)$ to be the order of the smallest TAV group of $K$, 
and call it the \textit{TAV order} of $K$ in short (we call it the \textit{minimal order} of $K$ in \cite{MS22-1}). Here, we note that the smallest TAV group might be not unique for a non-fibered knot $K$ in general (see Corollary \ref{cor:273}). 
For a fibered knot $K$, we set $\ord(K)=+\infty$, because its twisted Alexander polynomial is monic 
(see \cite{Cha03-1}, \cite{FK06-1}, \cite{GKM05-1}), 
and hence never vanishes. Hereafter, by abuse of terminology, 
we also call a finite group $G$ a \textit{TAV group} if $G$ is a TAV group of some knot $K$. 

In our previous paper \cite{IMS23-1}, we provide a characterization of a TAV group. In order to state the result precisely, 
we recall some terminologies in group theory. 
The \textit{weight} of a group $G$, denoted by $w (G)$, is the smallest integer $n$ such that $G$ is the normal closure of $n$ elements. In order to remove ambiguity, we set $w(\{e\})=0$. For a group $G$, there are a knot $K$ and an epimorphism $f\colon G(K)\to G$ if and only if $G$ is finitely generated and $w(G)\leq1$ (see \cite{GA75-1}, \cite{Johnson80-1}). 
A finite group $G$ is a \textit{$p$-group} if and only if the order $|G|$ is a power of a prime number. 
We simply denote the commutator subgroup $[G,G]$ of a group $G$ by $G'$. Then, we have:

\begin{theorem}[{\cite[Theorem~1.5]{IMS23-1}}]\label{thm:main-4}
A finite group $G$ is a TAV group if and only if 
$w(G)=1$ and $G'$ is not a $p$-group. 
\end{theorem}

For example, we easily see that any $p$-group is not a TAV group. Using Theorem \ref{thm:main-4} and Proposition \ref{pro:weight}(i) (see \cite[Corollaries 5.17, 5.18]{Chiodo14-1}), we can provide a list of TAV groups of order less than $201$. For more information, see Tables 1 to 3 in the appendix. 

Let $\mathcal{K}$ be the set of isotopy classes of oriented knots in $S^3$, and $\mathcal{N}\subset\mathcal{K}$ the subset consisting of non-fibered knots. Then, the TAV order of a knot induces a function 
$\ord \colon \mathcal{K}\to\mathbb{N}\cup\{+\infty\}$. In \cite{IMS23-1}, we also show some properties of $\ord(K)$ as follows: 

\begin{theorem}[{\cite[Theorem~1.2]{IMS23-1}}]\label{thm:main-1}
The TAV order $\ord\colon\mathcal{K}\to\mathbb{N}\cup\{+\infty\}$ has the following properties: 
\begin{itemize}
\item[(i)]
For any knot $K\in\mathcal{K}$, $\ord(K)\geq24$ holds. 
\item[(ii)]
The restriction $\ord|_{\mathcal{N}}\colon\mathcal{N}\to\mathbb{N}$ is unbounded. 
\item[(iii)]
For the connected sum $K_1\#K_2$ of two knots $K_1,K_2\in\mathcal{K}$, it holds that $\ord(K_1\#K_2)=\min\{\ord(K_1),\ord(K_2)\}$.
\item[(iv)]
If there is an epimorphism from $G(K_1)$ to $G(K_2)$, $\ord(K_1)\leq\ord(K_2)$ holds.
\item[(v)]
For a periodic knot $K$ and its quotient knot $K'$, 
$\ord(K)\leq\ord(K')$ holds. 
\item[(vi)]
If there is a proper degree one map $E_K\to E_{K'}$, 
where $E_K$ denotes the exterior of a knot $K$, 
then 
$\ord(K)\leq\ord(K')$ holds. 
\end{itemize}
\end{theorem}

\begin{remark}\label{rmk:mirror}
Throughout this paper, we do not distinguish a knot $K$ from its mirror image $K^*$, since their knot groups are isomorphic and hence they share the TAV order. We also note that $\ord(K)$ does not depend on an orientation of $K$. 
\end{remark}

For a given knot $K$, it seems difficult to determine the TAV order $\ord(K)$ explicitly. However, we can do it for some non-fibered knots by computer-aided calculation. Here, we adopt Rolfsen's table \cite{Rolfsen76-1} to represent a prime knot with $10$ or fewer crossings. 

\begin{theorem}\label{thm:main-2}
For any prime knot $K$ with $10$ or fewer crossings, we have 
\begin{itemize}
\item[(i)]
$\ord(K)=24$, if $K=9_{35}, 9_{46}$,
\item[(ii)]
$\ord(K)=60$, if $K=10_{67}, 10_{120}, 10_{146}$,
\item[(iii)]
$\ord(K)=96$, if $K=10_{166}$,
\item[(iv)]
$\ord(K)=120$, if $K=8_{15}, 9_{25}, 9_{39}, 9_{41}, 9_{49}, 10_{58}$,
\item[(v)]
$\ord(K)=168$, if $K= 7_4, 8_3$,
\item[(vi)]
$\ord(K)\geq201$, otherwise.
\hspace{3.1cm} 
\end{itemize}
\end{theorem}

Actually, we showed the statements (i), (ii) in \cite[Theorem 3.2]{MS22-1}, and (iii), (iv) in \cite[Theorem 1.3]{IMS23-1}. The new statement (v) can be shown by the similar way with the aid of a computer. As for (vi), we can check that for the other knots and for any TAV groups of order less than $201$ (see Tables 1 to 3) the twisted Alexander polynomials associated with the regular representations of these TAV groups do not vanish.  

The purpose of this paper is to explore several properties of TAV groups and the TAV order function. In particular, we discuss the following fundamental questions: 

\begin{itemize}
\item[(A)]
Find a non-fibered knot $K$ admitting a given TAV group $G$.

\item[(B)]
Let $\mathcal{V}(G)$ be the set of knots which admit $G$ as a TAV group. Then, for a given TAV group $G$, research basic properties of $\mathcal{V}(G)$. 

\item[(C)]
Determine the image of the TAV order function $\ord|_{\mathcal{N}}\colon\mathcal{N}\to \N$. 
\end{itemize}

By tracking the proof of Theorem \ref{thm:main-4}, we can theoretically find such a knot in (A), but it seems excessively impractical. As an accessible way to find a concrete non-fibered knot $K$ admitting a given TAV group $G$, we can adopt the \textit{satellite knot construction}. See Subsection \ref{subsec:3.0} for precise description of the knot. 

By using Theorem \ref{thm:main-1}(iii), it is easy to see that there are infinitely many \textit{composite} knots with the TAV orders appeared in Theorem \ref{thm:main-2}(i) to (v). However it is unclear that we can find infinitely many \textit{prime} knots with these TAV orders. 
As shown in \cite[Corollary 1.4]{IMS23-1}, there are infinitely many prime knots $K\in\mathcal{N}$ with $\ord(K)=24$. As for the question (B), we can extend the result as follows: 

\begin{theorem}\label{thm:main-5}
For any TAV group $G$, there are infinitely many hyperbolic knots that admit $G$ as a TAV group. Moreover, if $G$ is realized as a smallest TAV group of a knot, then infinitely many hyperbolic knots also admit $G$ as a smallest TAV group. 
\end{theorem}

Here, a knot $K$ in $S^3$ is \textit{hyperbolic} if the complement $S^3\setminus K$ admits the complete hyperbolic metric of finite volume. We remark that a hyperbolic knot is prime. 

At this point, we do not know how the TAV order $\ord(K)$ behaves when we restrict ourselves to special classes of knots, e.g., knots with a given bridge number, genus, braid index, crossing number, and so on. As an example, we can show the following interesting property of $\ord(K)$ for $2$-bridge knots: 

\begin{theorem}\label{thm:main-7}
The symmetric group $S_4$ is not a TAV group of any $2$-bridge knot. In other words, $\mathcal{V}(S_4)\cap \{\text{$2$-bridge knots}\}=\emptyset$. 
\end{theorem}

In particular, since the symmetric group $S_4$ is the smallest TAV group, $\ord(K)>24$ holds for every $2$-bridge knot $K$ (see Corollary \ref{cor:S_4}). 
Moreover, $S_4$ can be a TAV group of a non-fibered knot with any bridge number greater than two, the TAV order $\ord(K)$ allows us to distinguish the class of $2$-bridge knots from the other classes of knots (see Remark \ref{rmk:S_4}). 

Next, let us see the question (C); i.e. the realization problem of the TAV order $\ord(K)$.  In principle, we can make a list of the TAV groups and its orders using Theorem \ref{thm:main-4}.  However, it is not so easy to clarify whether a given order is realized as the TAV order of a knot. Namely, the determination of the image of the TAV order function $\ord|_{\mathcal{N}}\colon \mathcal{N}\to \N$ is an important question in our research. 
As a partial answer to the question, we can show the following: 

\begin{theorem}\label{thm:main-6}
Let $G$ be a TAV group of order $pqr$, where $p,q,r$ are primes.
Then there exist infinitely many hyperbolic knots that admit $G$ as a smallest TAV group;
in particular, there exist infinitely many hyperbolic knots $K$ with $\ord(K) = pqr$.
\end{theorem}

We note that the dihedral group $D_{qr}$ of order $2qr$ is a typical example of Theorem \ref{thm:main-6}.  Using this, we can answer one of the problems proposed in \cite[Section 5]{IMS23-1} (see Corollary \ref{cor:D_15}). 

Regarding the question (C), 
a relationship between twisted Alexander polynomials for a group extension is 
useful for studying both TAV groups and TAV orders. Let us consider a central extension 
$1\rightarrow C_n \rightarrow G_{k,n} \overset{\pr}{\longrightarrow} G_{k,1} \rightarrow 1$ by a finite cyclic group $C_n$, 
where $G_{k,n}$ is defined as the pull-back of two epimorphisms $\pi\colon G_{k,1}\to C_k$; the abelianization homomorphism (hencce the weight of $G_{k,n}$ is one), and $\pi_k\colon C_{kn}\to C_k$; a natural projection, where $k$ and $n$ are positive integers (see Section \ref{sec:extension} for details). 
We denote the regular representation of the group $G_{k,n}$ by $\rho_n\colon G_{k,n} \to\mathrm{GL}(kn|H|,\C)$, where $H$ is the kernel of $\pi$. 
Let $K$ be a knot and assume that there exists an epimorphism $f_n\colon G(K) \to G_{k,n}$. 
Then we have the following formula for the twisted Alexander polynomial associated with $\rho_n\circ f_n$: 

\begin{theorem}\label{thm:main-9}
$\displaystyle{
\D_K^{\rho_n\circ f_n}(t)
=
\prod_{l=0}^{n-1}
\D_K^{\rho_1\circ \pr \circ f_n}\big(e^{\frac{2l\pi i}{kn}}t\big).}$
\end{theorem}

Using the above formula, we can obtain some important information on the TAV order function $\ord|_{\mathcal{N}}\colon \mathcal{N}\to \N$ (see Example \ref{ex:dic33}). That is, there exists a positive integer which is the order of a TAV group, but never realize the TAV order of any non-fibered knot. In other words, the set of the orders of the TAV groups differs from $\mathrm{Im}\,\ord|_{\mathcal{N}}$, which gives a first step to settle the question (C). 

This paper is organized as follows. 
In Section \ref{sec:2}, 
we quickly recall the definition and one of vanishing criteria of twisted Alexander polynomials, and provide several examples of both TAV groups and non-TAV groups. 
Moreover, we explain a basic property of TAV groups under taking their quotient groups.  
In Section \ref{sec:4}, we construct a knot admitting a given TAV group, and prove Theorems \ref{thm:main-5} and \ref{thm:main-7}. 
In Section \ref{section:pqr}, we give a proof of Theorem \ref{thm:main-6} and discuss the uniqueness of the smallest TAV group for a given non-fibered knot. 
In Section \ref{sec:extension}, we provide a formula for the twisted Alexander polynomial for a central extension of a finite group (not necessarily TAV group), and show the existence of TAV groups whose orders cannot realize TAV orders of any knots. In Section \ref{sec:5}, we discuss some questions on TAV groups and their numerical invariants, and examine a relationship between the TAV order and the invariant coming from the circle valued Morse function (the Morse-Novikov number). Finally, in the appendix of this paper, we provide a list of the TAV groups of order less than $201$. 

\section{Preliminaries}\label{sec:2}

\subsection{Twisted Alexander polynomials}\label{subsec:2.1}

Let $X$ be a connected finite CW complex, 
$\phi\in H^1(X;\Z)=\mathrm{Hom}(\pi_1(X),\Z)$, 
and $\rho\colon\pi_1(X)\to \mathrm{GL}(n, R)$ a homomorphism to a general linear group over a Noetherian unique factorization domain $R$. 
Define a right $\Z[\pi_1(X)]$-module structure on 
$R^n\otimes_\Z\Z[t^{\pm1}]=R[t^{\pm1}]^n$ as follows: 
$$
(v\otimes p)\cdot g=(v\cdot \rho(g))\otimes(p\cdot t^{\phi(g)}),
$$
where 
$g\in \pi_1(X)$ and $v\otimes p\in R^n\otimes_\Z\Z[t^{\pm1}]$. 
Here, we view $R^n$ as row vectors. 
Taking tensor product, we obtain a homomorphism 
$\rho\otimes\phi\colon\pi_1(X)\to \mathrm{GL}(n,R[t^{\pm1}])$. 

We denote by $\tilde{X}$ the universal covering of $X$, 
and use the homomorphism 
$\rho\otimes\phi$ to regard 
$R[t^{\pm1}]^n$ as a right $\Z[\pi_1(X)]$-module. 
The chain complex $C_*(\tilde{X})$ is a left 
$\Z[\pi_1(X)]$-module via deck transformations. 
We can therefore consider the tensor products
$$
C_*(X;R[t^{\pm1}]^n)
:=R[t^{\pm1}]^n\otimes_{\Z[\pi_1(X)]}C_*(\tilde{X}),
$$ 
which form a chain complex of $R[t^{\pm1}]$-modules. 
We then consider the $R[t^{\pm1}]$-modules 
$H_*(X;R[t^{\pm1}]^n)
:=H_*(C_*(X;R[t^{\pm1}]^n))$. 

Since $X$ is compact and $R[t^{\pm1}]$ is Noetherian, 
these modules are finitely presented over $R[t^{\pm1}]$. 
We then define the \textit{twisted Alexander polynomial} of $(X,\phi,\rho)$ 
to be the order of $H_1(X;R[t^{\pm1}]^n)$ 
as a left $R[t^{\pm1}]$-module. 
We will denote it as $\D_{X,\phi}^\rho(t)\in R[t^{\pm1}]$, 
and note that $\D_{X,\phi}^\rho(t)$ is well defined up to multiplication by a unit in $R[t^{\pm1}]$. See \cite{FV11-1} for other basic properties of twisted Alexander polynomials. 

For a homomorphism $f \colon \pi_1(X) \to G$ to a finite group $G$, we get the representation
$$
\pi_1(X)\overset{f}{\longrightarrow} G
\overset{\rho}{\longrightarrow} \mathrm{Aut}_\Z(\Z[G]),
$$ 
where the second map is given by the right multiplication. 
We can also identify $\mathrm{Aut}_\Z(\Z[G])$ with 
$\mathrm{GL}(|G|,\Z)$, and 
obtain the corresponding twisted Alexander polynomial $\Delta_{X, \phi}^{\rho \circ f}(t)$. Then, the vanishing of $\Delta_{X, \phi}^{\rho \circ f}(t)$ is characterized as follows:

\begin{theorem}[{\cite[Theorem~4.2]{IMS23-1}}]\label{lifting-thm}
The twisted Alexander polynomial $\Delta^{\rho \circ f}_{X, \phi}(t)$ is zero if and only if there exists a nontrivial lift $\tilde{f} \colon \pi_1(X) \to \mathbb{Z}[G \times \mathbb{Z}] \rtimes (G \times \mathbb{Z})$, where $G \times \mathbb{Z}$ acts on $\mathbb{Z}[G \times \mathbb{Z}]$ by the left multiplication, of the homomorphism $f \times \phi \colon \pi_1(X) \to G \times \mathbb{Z}$, i.e., a group homomorphism $\tilde{f}$ such that $p_{G \times \mathbb{Z}} \circ \tilde{f} = f \times \phi$ and ${\rm Im}\; \tilde{f} \cap (\mathbb{Z}[G \times \mathbb{Z}] \times \{(e, 0)\}) \neq \{(0; e, 0)\}$, where $p_{G \times \mathbb{Z}} \colon \mathbb{Z}[G \times \mathbb{Z}] \rtimes (G \times \mathbb{Z}) \to G \times \mathbb{Z}$ is the projection.
\end{theorem}

In this paper, we consider a knot $K$ in the $3$-sphere $S^3$ 
and let $E_K=S^3\setminus \nu(K)$, where $\nu(K)$ denotes an open tubular neighborhood of $K$. 
We denote $\pi_1(E_K)$ by $G(K)$, 
and call it the \textit{knot group} of $K$. 
If $f\colon\G\to G$ is an epimorphism to a finite group $G$, 
we obtain the twisted Alexander polynomial 
$\D_{E_K,\phi}^{\rho\circ f}(t)$. 
For the \textit{abelianization} homomorphism 
$\phi \colon G(K)\to H_1(E_K;\Z)\cong\Z$, 
we drop $\phi$ from the notation and use $\D_K^{\rho\circ f}(t)$ for simplicity. 

\begin{remark}\label{rmk:torsion}
It is known that 
$\D_{K}^{\rho\circ f}(t)\not=0$ if and only if 
$H_1(E_K;\Q[G][t^{\pm1}])=H_1(E_K;\Q[t^{\pm1}]^{|G|})$ is $\Q[t^{\pm1}]$-torsion, 
namely, 
$\mathrm{rank}_\Z\,H_1(E_K;\Z[t^{\pm1}]^{|G|})$ is finite 
(see \cite[Remark~4.5]{Turaev01-1}). 
\end{remark}

\subsection{Twisted Alexander vanishing groups}\label{subsec:2.2}

For reader's convenience, we exhibit here some examples of both TAV groups and non-TAV groups. To this end, let us recall the definition of a TAV group of a knot from the introduction. We call a finite group $G$ a \textit{twisted Alexander vanishing (TAV) group of a knot $K$} in $S^3$ if there exists an epimorphism $f$ of the knot group $G(K)$ onto $G$ such that the twisted Alexander polynomial associated with  the composition of $f$ and the regular representation of $G$ is zero. By abuse of terminology, we also simply call a finite group $G$ a TAV group if $G$ is a TAV group of some knot $K$. 

In view of Theorem \ref{thm:main-4}, for a finite group $G$ such that $w(G)=1$ and the commutator subgroup $G'$ is not a $p$-group, there exist a non-fibered knot $K$ and an epimorphism $f \colon G(K) \to G$ such that the corresponding twisted Alexander polynomial $\D_{K}^{\rho\circ f}(t)$ is zero. 

First of all, a finite \textit{abelian group} $G$ of weight one is not a TAV group, because it
is cyclic, i.e. $G=C_n$. In this case, for an epimorphism $f\colon G(K)\to G$, we have $\D_K^{\rho\circ f}(t)=\prod_{j=1}^{n}\D_K(\alpha^jt)$, where $\alpha\in \C$ is a primitive $n$-th root of unity and $\D_K(t)$ is the Alexander polynomial of $K$ (see \cite[Example 2.6]{IMS23-1}). 

Next, as mentioned in the introduction, every \textit{$p$-group} is not a TAV group. Since a finite $p$-group $G$ is nilpotent, $G$ is cyclic when $w(G)=1$. Hence, $G$ is not a TAV group. 

On the other hand, the following three kinds of groups are TAV groups: 

\begin{itemize}
\item
The \textit{alternating group} $A_n~(n\geq5)$. 

\item
The \textit{symmetric group} $S_n=A_n\rtimes C_2~(n\geq4)$.

\item
The \textit{dihedral group} $D_n=C_n\rtimes C_2$ of order $2n$, where $n$ is odd and not a power of an odd prime number. 
\end{itemize}

For $n\geq2$, the \textit{dicyclic group} $\mathrm{Dic}_n$ of order $4n$ is defined by the presentation: 
$$
\mathrm{Dic}_n=
\la
x,y\,|\,x^{2n},y^2x^{-n},yxy^{-1}x
\ra,
$$
which can be considered as an extension of $C_2$ by $C_{2n}$. Namely, that group fits the following short exact sequence: 
$$
1\to C_{2n}\to \mathrm{Dic}_n\to C_2\to1.
$$
Moreover, $\mathrm{Dic}_n$ is often called the \textit{binary dihedral group}, because $\mathrm{Dic}_n/\la y^2\ra$ is isomorphic to $D_n$. If $n$ is even, $\mathrm{Dic}_n$ is not weight one, and hence not a TAV group. Then, we can show the following: 

\begin{itemize}
\item
The \textit{dicyclic group} $\mathrm{Dic}_n$ of order $4n$, where 
$n$ is odd and not a power of an odd prime number, is a TAV group. 
\end{itemize}

The following characterization of the weight of a finite or soluble group is known in group theory (see \cite[Corollaries 5.17, 5.18]{Chiodo14-1}). 

\begin{proposition}\label{pro:weight}
\begin{itemize}
\item[(i)]
Let $n>1$ be an integer and let $G$ be a finite or soluble group. Then $w(G)=n$ if and only if $w(G/G') = n$, and $w(G) \leq 1$ if and only if 
$w (G/G') \leq 1$. 
\item[(ii)]
Nontrivial finite (or soluble) perfect groups have weight one. 
\end{itemize}
\end{proposition}

Using Proposition \ref{pro:weight}, we easily see that  
\begin{itemize}
\item
Every finite \textit{perfect} group is a TAV group. In particular, every finite \textit{simple} group is a TAV group as well.
\end{itemize} 

There are three nontrivial perfect groups of order less than $201$, which are $A_5, \mathrm{SL}_2(\mathbb{F}_5)$, and $\mathrm{GL}_3(\mathbb{F}_2)$ (see \cite{GN} and Tables 1 and 2 in the appendix). 

\subsection{Quotient groups}\label{subsec:2.3}

Let $G$ be a finite group, $H\triangleleft G$ a normal subgroup, and $\pr\colon G\to G/H$ the projection. We denote the regular representation of $G$ by $\rho_G$. Then, as for the vanishing of the twisted Alexander polynomial associated with the quotient group, we have the following. 

\begin{proposition}\label{pro:quotient}
If there are two epimorphisms $\overline{f}\colon G(K)\to G/H$ suct that $\D_K^{\rho_{G/H}\circ \overline{f}}(t)=0$ and $f\colon G(K)\to G$ such that $\pr\circ f=\overline{f}$, then $\D_K^{\rho_G\circ f}(t)=0$. 
\end{proposition}

We then see that for any non-fibered knot $K$ there are infinitely many TAV groups of $K$, because of Theorem \ref{thm:FV} and Proposition \ref{pro:quotient}. In fact, we only have to take the direct product of a given TAV group  and a finite cyclic group so that the weight of the direct product is one. The above proposition follows from the following lemma. We omit its proof. 

\begin{lemma}\label{lem:quotient}
Let $\rho_G\colon G\to \mathrm{GL}(|G|,\C)$ and $\rho_{G/H}\colon G/H\to \mathrm{GL}(|G/H|,\C)$ be the regular representations respectively. If there is an epimorphism $f\colon G(K)\to G$, then $\D_K^{\tilde{\rho}_{G/H}\circ f}(t)$ divides $\D_K^{\rho_G\circ f}(t)$, where $\tilde{\rho}_{G/H}=\rho_{G/H}\circ\pr$.
\end{lemma}

The next example is a consequene of Proposition \ref{pro:quotient} (see Table 1). 

\begin{example}
Let $G$ be one of three finite groups $\mathrm{CSU}_2(\mathbb{F}_3)$, $\mathrm{GL}_2(\mathbb{F}_3)$, $A_4\rtimes C_4$ of order $48$, where $\mathbb{F}_3$ denotes a finite field of characteristic $3$. We omit the precise definitions of these groups (see \cite{GN}). Since they admit the symmetric group $S_4$ as a qutient group, and non-fibered knots $K=9_{35},9_{46}$ have $S_4$ as the smallest TAV group (see Theorem \ref{thm:main-2}(i) and Table 1), we can conclude that these three groups are TAV groups of $K$. We can also find this kind of examples in Tables 1 to 3. 
\end{example}

\section{Some properties of $\ord(K)$}\label{sec:4}

In this section, we discuss some properties of the TAV order $\ord(K)$. Firstly, we explain how to construct knots admitting a given TAV group. 
Secondly, we show the existence of infinitely many hyperbolic non-fibered knots for a given TAV group. Thirdly, we exhibit a notable property of $\ord(K)$ for $2$-bridge knots. 

\subsection{Knots admitting a given TAV group}\label{subsec:3.0}
Let $G$ be a finite group with $w(G) = 1$ and assume that $G'$ is not a $p$-group. By Theorem \ref{thm:main-4} (\cite[Theorem~1.5]{IMS23-1}), $G$ is a TAV group of some knot. We can theoretically find such a knot by tracking the proof given in \cite{IMS23-1}, but it seems an impractically hard work and the resultant knot would be too complicated to examine. In this subsection, we give a more practical recipe to construct a knot admitting a given TAV group.

In order to explain the procedure, we recall the satellite construction of knots following \cite{CS16-1}. Let $K$ be a knot and $\alpha$ an unknotted loop disjoint from $K$. For another knot $J$, we glue the exteriors $E_\alpha$ of $\alpha$ and $E_J$ of $J$ by an orientation-reversing homeomorphism between the boundary tori so that a meridian and a preferred longitude of $\alpha$ are respectively identified with a preferred longitude and a meridian of $J$. Since the obtained $3$-manifold is homeomorphic to $S^3$, we can regard $K \subset E_\alpha$ as a new knot in $S^3$, which we denote by $K(\alpha, J)$. See Figure \ref{fig:satellite} in Section \ref{section:pqr} for an example, where $K= T(2, 15)$ and $J = T(3,5)$. 

Let $i \colon E_{K \cup \alpha} \to E_K$ and $j \colon E_{K \cup \alpha} \to E_{K(\alpha, J)}$ be the inclusion maps, which induce group homomorphisms $i_* \colon G(K \cup \alpha) \to G(K)$ and $j_* \colon G(K \cup \alpha) \to G(K(\alpha, J))$, where $E_{K\cup\alpha}=S^3\setminus\nu(K\cup\alpha)$ and $G(K\cup\alpha)=\pi_1(E_{K\cup\alpha})$. By \cite[Lemma 4.1]{CS16-1}, there exists a group homomorphism $\psi: G(K(\alpha, J)) \to G(K)$ such that $i_* = \psi \circ j_*$: $\psi$ is constructed by gluing the inclusion homomorphism $i_* \colon G(K \cup \alpha) \to G(K)$ and $\alpha_* \circ \phi_J \colon G(J) \to G(K)$, where $\phi_J \colon G(J) \to \mathbb{Z}$ is the abelianization map and $\alpha_* \colon \mathbb{Z} \to G(K)$ is the group homomorphism that sends $1 \in \mathbb{Z}$ to the element represented by the loop $\alpha$.

Let us see how to construct a knot admitting a given TAV group. Let $G$ be any TAV group. Since $G'$ has non-prime-power order, by \cite[Lemma~1]{gru}, $G'$ contains at least one of (i) a cyclic group of non-prime-power order or (ii) a non-abelian group of weight one; 
take such a subgroup and denote it by $H$.

Since $w(G) = 1$, there exists a knot $K$ admitting an epimorphism $f_0\colon G(K) \to G$ onto $G$. Since $H$ is also of weight one, 
we can take a loop $\alpha \subset E_K$ that bounds a disc in $S^3$ such that ${\rm lk}(K, \alpha) = 0$ and $f_0(\alpha)$ is contained in $H$ and normally generates $H$ (in $H$). In the case (i), we take a knot $J$ with $\Delta_J(e^{2 \pi i/ |H|}) = 0$ and define $f = f_0 \circ \psi \colon G(K(\alpha, J)) \to G$. In the case (ii), we take a knot $J$ that admits an epimorphism $f_J \colon G(J) \to H$ such that $f_J(m) = f_0(\alpha)$ and $f_J(\ell) = e$, where $(m, \ell)$ is a meridian-longitude pair of $J$, and then obtain a homomorphism $f \colon G(K(\alpha, J)) \to G$ by gluing $f_0$ and $f_J$.

\begin{theorem}\label{const-thm}
We have $\Delta_{K(\alpha, J)}^{\rho \circ f}(t) = 0$. In particular, $G$ is a TAV group of the knot $K(\alpha, J)$.
\end{theorem}

In the two propositions below, we examine when $\Delta_{K(\alpha, J)}^{\rho \circ f}(t) = 0$ and then Theorem \ref{const-thm} follows from them immediately: The case (i) follows from Proposition \ref{const-prop-1} and the case (ii) from Proposition \ref{const-prop-2}.

\begin{remark}
By \cite[Lemma~1]{gru}, we can take $H$ as follows: (i) a cyclic group of order $pq$ with distinct primes $p, q$ or (ii) a non-abelian metabelian group $C_p^n \rtimes C_q$ of weight one, 
where $p, q$ are distinct primes and $n$ is a positive integer. If we choose $H$ in these ways, we can take the torus knot $T(p,q)$ as $J$.
\end{remark}

Let $G$ be a finite group and $f \colon G(K(\alpha, J)) \to G$ a homomorphism. Since $E_{K(\alpha, J)} = E_{K \cup \alpha} \cup E_J$, we have the restriction homomorphism $f_J\colon G(J) \to G$. The homomorphism $f$ factors through $\psi$, i.e., there exists a homomorphism $f_0 \colon G(K) \to G$ such that $f = f_0 \circ \psi$, if and only if the image of $f_J$ is cyclic. 

\begin{proposition}\label{const-prop-1}
Suppose that the linking number ${\rm lk}(K, \alpha)$ equals zero. We assume that $f$ factors through $\psi$, and let $d$ be the order of the cyclic group $f_J(G(J))$. Then, $\Delta^{\rho \circ f}_{K(\alpha, J)}(t) = 0$ if and only if $\Delta^{\rho \circ f_0}_K (t) = 0$ or $\Delta_J(e^{2 \pi i /d} ) = 0$, where $\Delta_J (t)$ is the Alexander polynomial of $J$.
\end{proposition}
\begin{proof}
By \cite[Theorem 3.1]{KSW05-1}, there exists a polynomial $h(t) \in \mathbb{Z}[t^{\pm 1}]$ such that $\Delta^{\rho \circ f}_{K(\alpha, J)}(t) = \Delta^{\rho \circ f_0}_K (t) \cdot h(t)$; since $\Delta^{\rho \circ f_0}_K (t) = 0$ implies that $\Delta^{\rho \circ f}_{K(\alpha, J)}(t) = 0$, we assume $\Delta^{\rho \circ f_0}_K (t) \neq 0$ and show that $\Delta^{\rho \circ f}_{K(\alpha, J)}(t) = 0$ if and only if $\Delta_J(e^{2 \pi i /d} ) = 0$.

Let $\tilde{E}$ be the covering space of $E_{K(\alpha, J)}$ corresponding to $f \times \phi \colon G(K(\alpha, J)) \to G \times \mathbb{Z}$, i.e., let $\tilde{E}_{K(\alpha, J)}$ be the universal covering space of $E_{K(\alpha,J)}$ 
and define
$$\tilde{E} = (G \times \mathbb{Z}) \times_{f \times \phi} \tilde{E}_{K(\alpha, J)},$$
where $G$ and $\mathbb{Z}$ are equipped with the discrete topologies; remark that $\tilde{E}$ is not necessarily connected. It is sufficient to show that $H_1(\tilde{E})$ has infinite rank over $\mathbb{Z}$ (see Remark \ref{rmk:torsion}). Let $\pi_{K(\alpha, J)} \colon \tilde{E} \to E_{K(\alpha, J)}$ denote the covering map, and set
$$\tilde{E}_{K(\alpha, J)} = \pi_{K(\alpha, J)}^{-1}(E_{K \cup \alpha}), \qquad \tilde{E}_J = \pi_{K(\alpha, J)}^{-1}(E_J),$$
recalling that $E_{K(\alpha, J)} = E_{K \cup \alpha} \cup E_J$ as above.

Since $\tilde{E} = \tilde{E}_{K\cup \alpha} \cup \tilde{E}_J$, we have the Mayer-Vietoris exact sequence
$$\xymatrix@R=0cm{\cdots \ar[r] & H_1(\tilde{E}_{K \cup \alpha} \cap \tilde{E}_J) \ar[r]^-{\iota = (\iota_1, \iota_2)} & H_1(\tilde{E}_{K \cup \alpha}) \oplus H_1(\tilde{E}_J) \ar[r] & H_1(\tilde{E}) & \\
\ar[r]^-\partial & H_0(\tilde{E}_{K \cup \alpha} \cap \tilde{E}_J) \ar[r] & H_0(\tilde{E}_{K \cup \alpha}) \oplus H_0(\tilde{E}_J) \ar[r] & H_0(\tilde{E}) \ar[r] & 0.}$$
Because ${\rm lk}(K, \alpha) = 0$ and $f_J(G(J))$ is cyclic of order $d$, each connected component of $\tilde{E}_J$ is the $d$-fold cyclic covering space of $E_J$; in particular, the boundary is also connected and hence the inclusion homomorphism of $\tilde{E}_{K \cup \alpha} \cap \tilde{E}_J \subset \tilde{E}_J$ at dimension $0$ is injective. This implies that the connecting homomorphism $\partial\colon H_1(\tilde{E}) \to H_0(\tilde{E}_{K \cup \alpha} \cap \tilde{E}_J)$ equals zero. Thus, $H_1(\tilde{E})$ is isomorphic to the cokernel of the inclusion homomorphism $\iota \colon H_1(\tilde{E}_{K \cup \alpha} \cap \tilde{E}_J) \to H_1(\tilde{E}_{K \cup \alpha}) \oplus H_1(\tilde{E}_J)$.

Let $m \subset \partial E_J$ be a meridional loop and $\ell \subset \partial E_J$ a longitudinal loop with ${\rm lk}(J, \ell) = 0$. Let $M, L \subset H_1(\tilde{E}_{K\cup\alpha} \cap \tilde{E}_J)$ be the subspaces generated by the homology classes represented by the lifts of $m^d,\ell$, respectively. Since $\tilde{E}_{K\cup\alpha} \cap \tilde{E}_J = \partial E_J$, $H_1(\tilde{E}_{K\cup\alpha} \cap \tilde{E}_J) = M \oplus L$. Because each connected component of $\tilde{E}_J$ is the $d$-fold cyclic covering space of $E_J$, $\iota_2|_L \colon L \to H_1(\tilde{E}_J)$ is zero and $\iota_2|_M \colon M \to H_1(\tilde{E}_J)$ is injective. Hence ${\rm Coker} \; \iota$ contains a subspace isomorphic to $H_1(\tilde{E}_{K \cup \alpha})/\iota_1(L) \cong H_1(\tilde{E}_K)$, where $\tilde{E}_K$ is the covering space of $E_K$ corresponding to $f_0 \times \phi\colon G(K) \to G \times \mathbb{Z}$, and the quotient of ${\rm Coker}\;\iota$ over the subspace is isomorphic to the homology group $H_1(\widehat{E}_J)$ of $\widehat{E}_J$, the $3$-manifold obtained by attaching solid tori to the boundary components of $\tilde{E}_J$ so that the lifts of $m^d$ bound discs in $\widehat{E}_J$.

As we assume $\Delta^{\rho \circ f_0}_K (t) \neq 0$, the subspace isomorphic to $H_1(\tilde{E}_K)$ is of finite rank over $\mathbb{Z}$. Furthermore, because each connected component of $\widehat{E}_J$ is homeomorphic to the $d$-fold branched cyclic cover $B_J^d$ of $E_J$, $H_1(\widehat{E}_J)$ is isomorphic to the direct sum of infinitely many copies of $H_1(B_J^d)$. Therefore we have $\Delta^{\rho \circ f}_{K(\alpha, J)}(t) = 0$, i.e., ${\rm rank}_{\mathbb{Z}} \, H_1(\tilde{E}) = \infty$, if and only if $H_1(B_J^d)$ has positive rank, which is equivalent to the condition $\Delta_J(e^{2 \pi i/d}) = 0.$
\end{proof}

\begin{proposition}\label{const-prop-2}
Suppose that the linking number ${\rm lk}(K, \alpha)$ equals zero. If $f \colon G(K(\alpha, J)) \to G$ does not factor through $\psi$, $\Delta^{\rho \circ f}_{K(\alpha, J)}(t) = 0$.
\end{proposition}
\begin{proof}
By the Seifert-van Kampen theorem, the knot group $G(K(\alpha, J))$ is the amalgamated free product $G(K \cup \alpha) *_{\pi_1(\partial E_J)} G(J)$. Let $m, \ell \in G(J)$ denote the meridian and the preferred longitude of $J$, and let $d_1, d_2$ be the orders of $f(m)$ and $f(\ell)$, respectively. We set
$$\xi = \left(\sum_{j=0}^{d_1-1}(f(m)^j,0)\right) \left(\sum_{j=0}^{d_2-1} (f(\ell)^j,0)\right) \in \mathbb{Z}[G \times \mathbb{Z}],$$
and define homomorphisms $\tilde{f}_{K \cup \alpha} \colon G(K \cup \alpha) \to \mathbb{Z}[G \times \mathbb{Z}] \rtimes (G \times \mathbb{Z})$ and $\tilde{f}_J \colon G(J) \to \mathbb{Z}[G \times \mathbb{Z}] \rtimes (G \times \mathbb{Z})$ by
\begin{align*}
\tilde{f}_{K \cup \alpha}(u) &= (0; f(u), \phi(u)) & (u \in G(K \cup \alpha)),\\
\tilde{f}_J(u) &= \bigl(((f(u),0) - (e,0))\xi; f(u), 0\bigr) & (u \in G(J)).
\end{align*}
Because $m$ commutes with $\ell$, $((f(u),0) - (e,0))\xi = 0$ for $u \in \pi_1(\partial E_J)$. Therefore we obtain a homomorphism $\tilde{f} \colon G(K(\alpha, J)) \to \mathbb{Z}[G \times \mathbb{Z}] \rtimes (G \times \mathbb{Z})$ by gluing $\tilde{f}_{K \cup \alpha}$ and $\tilde{f}_J$.

By the definitions of $\tilde{f}_{K \cup \alpha}$ and $\tilde{f}_J$, $\tilde{f}$ is a lift of $f \times \phi$ (remark that $\phi$ is zero on $G(J)$ since ${\rm lk}(K, \alpha) = 0$). Furthermore we should remark that the image $f(E_J) \subset G$ is not generated by the two elements $f(m), f(\ell)$; if generated, $f(E_J)$ is an abelian group of weight one and hence cyclic, which contradicts the assumption that $f$ does not factor through $\psi$. Therefore we can take $u_1 \in G(J)$ such that $f(u_1) \not\in \langle f(m), f(\ell) \rangle$. Let $u_2 \in G(K \cup \alpha)$ be a meridional loop of $K$ and $d_3, d_4$ the orders of $f(u_1 u_2), f(u_2) \in G$. Then, we have
\begin{eqnarray*}
\tilde{f}((u_1u_2)^{d_3 d_4} u_2^{-d_3 d_4}) &=& \bigl(((f(u_1),0) - (e, 0))\xi; f(u_1 u_2), 1\bigr)^{d_3 d_4} (0; e, -d_3 d_4)\\
&=& \left( \sum_{j=0}^{d_3 d_4 -1} (f(u_1u_2)^j, j) ((f(u_1),0) - (e, 0))\xi; e, 0 \right).
\end{eqnarray*} 
Here, we note that the product in $\mathbb{Z}[G \times \mathbb{Z}] \rtimes (G \times \mathbb{Z})$ is given by $(\lambda;x,a)(\lambda';x',a')=(\lambda+(x,a)\lambda';xx',a+a')$, where $\lambda,\lambda'\in \Z[G\times \Z], x,x'\in G$, and $a,a'\in\Z$. 
Since $f(u_1) \not\in \langle f(m), f(\ell) \rangle$, $\sum_{j=0}^{d_3 d_4 -1} (f(u_1u_2)^j, j) ((f(u_1),0) - (e, 0))\xi \neq 0$. Therefore $\tilde{f}$ is a nontrivial lift and we have $\Delta^{\rho \circ f}_{K(\alpha, J)}(t) = 0$ by Theorem \ref{lifting-thm}.
\end{proof}

\begin{remark}
If ${\rm lk}(K, \alpha) \neq 0$, then $\Delta_{K(\alpha, J)}^{\rho \circ f}(t) = 0$ if and only if $\Delta_{K\cup \alpha}^{\rho \circ f}(t) = 0$ or $\Delta_J^{\rho \circ f}(t) = 0$. 
In particular, if $f$ maps a preferred longitude of $J$ to the identity, e.g., if $f$ factors through $\psi$, then $f$ induces a homomorphism $f_0 \colon G(K) \to G$ and the condition $\Delta_{K\cup \alpha}^{\rho \circ f}(t) = 0$ is equivalent to $\Delta_K^{\rho \circ f_0}(t) = 0$. A proof of these facts is similar to that of Proposition \ref{const-prop-1}; a different point is that $H_1(\tilde{E}_{K \cup \alpha} \cap \tilde{E}_J)$ is of finite rank if ${\rm lk}(K, \alpha) \neq 0$.
\end{remark}

\subsection{Existence of hyperbolic knots for a given TAV group}\label{sec:4.1}

In the previous subsection, we constructed a knot admitting a given TAV group. 
The purpose of this subsection is to prove the following theorem: 

\begin{theorem}[Theorem \ref{thm:main-5}]\label{thm:hyperbolic}
For any TAV group $G$, there are infinitely many hyperbolic knots that admit $G$ as a TAV group. Moreover, if $G$ is realized as a smallest TAV group of a knot, then infinitely many hyperbolic knots also admit $G$ as a smallest TAV group. 
\end{theorem}

\begin{proof}
Let $K$ be a knot admitting $G$ as a TAV group and take an $n$-braid $b$ that presents $K$. Here, we may assume that $b$ is a pseudo-Anosov element in the braid group $B_n$. For example, if $n$ is prime, this assumption is satisfied; the primeness implies the irreducibility, and $b$ is not periodic since the closure $K$ is not a torus link. For a positive integer $k$ coprime to $n$, let $K_k$ be the knot obtained by taking the closure of the braid $b^k$. We first show that $G$ is a TAV group of $K_k$, and under the assumption $\ord(K) = |G|$ we prove that $\ord(K_k) = \ord(K)$ holds if $k$ is coprime to the natural numbers less than or equal to $n\cdot\ord(K)^n$. Finally, we see that $K_k$ is hyperbolic if $k$ is sufficiently large.

Let $p_k \colon G(K_k) \to G(K)$ denote the quotient map. For a surjective homomorphism $f \colon G(K) \to G$ with $\Delta_K^{\rho \circ f}(t) = 0$, $\Delta_{K_k}^{\rho \circ (f \circ p_k)}(t) = 0$ since $\Delta_K^{\rho \circ f}(t)$ is a factor of $\Delta_{K_k}^{\rho \circ (f \circ p_k)}(t)$ by \cite[Theorem~1]{HLN06-1}; $G$ is a TAV group of $K_k$.

Let us consider the case $\ord(K) = |G|$. As shown above, we have $\ord(K_k)\leq\ord(K)$ for any $k$.
Let us assume that $k$ is coprime to the natural numbers less than or equal to $n\cdot\ord(K)^n$ and show the other inequality.

We first claim that any homomorphism from $G(K_k)$ to a group $H$ with $|H| < \ord(K)$ factors through the quotient map $p_k \colon G(K_k) \to G(K)$. To see this, we consider the Hurwitz action of $B_n$ on $H^n$. We should recall that there is a one-to-one correspondence between the fixed point set of $b \;\text{(resp. $b^k$)} \in B_n$ and the set of the group homomorphisms from $G(K)$ (resp. $G(K_k)$) to $H$. By the assumption, $k$ is coprime to $|H^n|!$ and hence is coprime to the order of $b$. Therefore the fixed point set of $b$ is equal to that of $b^k$; in other words, any homomorphism from $G(K_k)$ to $H$ factors through $G(K)$.

Thus, any homomorphism $f_k: G(K_k) \to H$ is the composite of the quotient map $p_k$ and a homomorphism $f: G(K) \to H$ if $|H| < \ord(K)$. By \cite[Theorem~3]{HLN06-1}, there exists a polynomial $F(t, s) \in \mathbb{C}[t^{\pm 1}, s^{\pm 1}]$ such that
$$\Delta_{K_k}^{\rho \circ f_k}(t) = \Delta_K^{\rho \circ f}(t) \prod_{j = 1}^{k-1} F(t, \zeta_k^j),$$
where $\zeta_k \in \mathbb{C}$ is a primitive $k$-th root of unity. As in the proof of \cite[Lemma~19]{HLN06-1}, $\Delta_K^{\rho \circ f}(t) \neq 0$ implies that $F(t, s) \neq 0$. Moreover, the width of $F(t, s)$ with respect to $s$ is less than or equal to $n|H|$ since $F(t, s)$ is the two-variable twisted Alexander polynomial $\Delta_{K \cup \alpha}^{\rho \circ f}(t, s)$ in the sense of \cite{Wada94-1}, where $\alpha$ is a trivial knot component around $b$. Since $k$ is coprime to any natural number less than or equal to $n|H|$, we have $F(t, \zeta_k^j) \neq 0$ for $j = 1, \dots, k-1$ and hence $\Delta_{K_k}^{\rho \circ f_k}(t) \neq 0$, which concludes that $H$ is not a TAV group of $K_k$ if $|H| < \ord(K)$, i.e., $\ord(K_k) \geq \ord(K)$.

Let us show that the knot $K_k$ is hyperbolic for sufficiently large $k$ using Thurston's hyperbolic Dehn surgery theorem. We regard $b$ as a self-homeomorphism of the $(n+1)$-punctured sphere, and construct a mapping torus $M$ with monodromy $b$. Then, $M$ becomes a hyperbolic $3$-manifold, 
because $b$ is a pseudo-Anosov element in $B_n$. If we take a trivial loop $\alpha$ around $b$, then $K\cup\alpha$ becomes a hyperbolic link. Hence, for a sufficiently large $k\in\N$, the hyperbolic Dehn surgery theorem implies that $M_{(\infty,k/0)}$ admits the hyperbolic structure. Here, $(\infty,k/0)$ denotes the attaching slopes for $K$ and $\alpha$ respectively. For example, $M_{(\infty,1/0)}$ denotes $S^3\setminus K$. 

When $k>1$ the loop $\alpha$ becomes a cone singularity, but if we take the $k$-fold cyclic branched covering along $\alpha$, we can obtain the usual hyperbolic structure.  Then, the resulting manifold coincides with $S^3\setminus K_k$, and hence, $K_k$ is a hyperbolic knot as desired. 
\end{proof}

Using Theorems \ref{thm:main-2} and \ref{thm:hyperbolic}, we have the following corollary (see \cite[Corollary 1.4]{IMS23-1} for the case $\ord(K)=24$). 

\begin{corollary}\label{cor:hyperbolic}
There are infinitely many hyperbolic knots $K\in\mathcal{N}$ with $\ord(K)=24,60,96,120$, or $168$.
\end{corollary}

Now let us consider the subset $\mathcal{V}(G)\subset\mathcal{K}$ for a given TAV group $G$: 
$$
\mathcal{V}(G):=
\{K\in\mathcal{K}\,|\,\text{$K$ admits $G$ as a TAV group}\}.
$$
By Theorem \ref{thm:hyperbolic}, we see that there exist infinitely many hyperbolic knots in $\mathcal{V}(G)$. Since the cardinality of $\mathcal{V}(G)$ is infinity for any TAV group $G$, it is of particular interest to determine which knots are contained in the set $\mathcal{V}(G)$, and to characterize 
$\mathcal{V}(G)$ for each TAV group $G$. To this end, we introduce a new concept associated with TAV groups. 

Let $c(K)$ be the crossing number of $K$, 
namely, the smallest number of crossings of any diagram of the knot $K$. 
We define the \textit{crossing number of a TAV group} $G$ to be the minimal crossing number of knots $K$ contained in $\mathcal{V}(G)$; 
$$
c(G):=\min_{K\in\mathcal{V}(G)}\{c(K)\}.
$$
We easily see that $c(G)\geq 5$ holds for any TAV group $G$, because the knots with four or fewer crossings are fibered.  

By computer-aided calculation, we can determine $c(G)$ for some TAV groups $G$ of order less than $201$ (see Tables 1 to 3). In the next subsection, we discuss another invariant of a TAV group $G$ called the bridge number of $G$. 

\subsection{TAV orders of 2-bridge knots}\label{sec:4.3}

The \textit{bridge number} $b(K)$ is an invariant of a knot $K$ defined as the minimal number of bridges required in all the possible bridge representations of a knot.

Let $K=S(b,a)$ be a $2$-bridge knot given by the Schubert normal form, where $a$ and $b$ are coprime integers $(|a|<b)$, and $b>0$ is odd (see \cite{BZH14-1}). It is known that $S(b,a)$ and $S(b',a')$ are equivalent if and ony if $b'=b$ and $a'\equiv a^{\pm1}$ (mod $b$), and that $S(b,-a)$ gives the mirror image of $S(b,a)$. 

The knot group of $K=S(b,a)$ has a presentation $G(K)=\langle u,v\, |\, w u =v w\rangle$ where 
$w=u^{\epsilon_1}v^{\epsilon_2}\cdots u^{\epsilon_{b-2}}v^{\epsilon_{b-1}}$ and 
$\epsilon_j=(-1)^{\left[\frac{a}{b}j\right]}$. 
Here 
we write $[r]$ for the greatest integer less than or equal to $r\in\Bbb{R}$. 
It is easily checked that 
$\epsilon_j=\epsilon_{b-j}$ holds. 
The above presentation is not unique for a $2$-bridge knot $K$ itself, 
but the existence of at least one such presentation follows from 
Wirtinger's algorithm applied to the Schubert normal form of $S(b,a)$. 
The generators $u$ and $v$ come from the two overpasses and present the 
meridian of $S(b,a)$ up to conjugation. 

In this subsection, we show the following theorem. 

\begin{theorem}[Theorem \ref{thm:main-7}]\label{thm:S_4}
The symmetric group $S_4$ is not a TAV group of any $2$-bridge knot. 
In other words, $\mathcal{V}(S_4)\cap \{\text{$2$-bridge knots}\}=\emptyset$. 
\end{theorem}

As shown in Table 1, the TAV group of order $24$ is only the symmetric group $S_4$. Hence, Theorems \ref{thm:main-1}(i) and \ref{thm:S_4} imply the following corollary. 

\begin{corollary}\label{cor:S_4}
For any $2$-bridge knot $K$, $\ord(K)>24$ holds. 
\end{corollary}

\begin{remark}\label{rmk:S_4}
We can show that for any integer $n\geq3$, there are infinitely many non-fibered 
$n$-bridge knots $K$ with $\ord(K)=24$. In fact, we can make such knots by connected sum. 
\end{remark}

We define the \textit{bridge number of a TAV group} $G$ to be the minimal bridge number of knots $K$ contained in $\mathcal{V}(G)$; $b(G):=\min_{K\in\mathcal{V}(G)}\{b(K)\}$. Since every nontrivial knot has bridge number at least two, $b(G)\geq 2$ holds for any TAV group $G$. As an immediate corollary of Theorem \ref{thm:S_4} and Remark \ref{rmk:S_4}, we have 

\begin{corollary}
$b(S_4)=3$.
\end{corollary}

\begin{proof}[Proof of Theorem \ref{thm:S_4}]
Let $K$ be a $2$-bridge knot. We fix a presentation 
$G(K)=\la u,v\,|\,wu=vw\ra$, where $w$ is a word in $u$ and $v$. 
The group $S_4$ has 
\begin{itemize}
\item[(1)] two $1$-dimensional irreducible representations, 
\item[(2)] one $2$-dimensional irreducible representation which is lifted from $S_3$, and 
\item[(3)] two $3$-dimensional irreducible representations. 
\end{itemize}
Non-vanishing of the twisted Alexander polynomial in the cases (1) and (2) is easy consequence of \cite[Example 2.6]{IMS23-1} and Theorem 1.3(i), so that we consider the $3$-dimensional irreducible representations.  

Let $\tau\colon S_4\to \mathrm{GL}(3,\Z)$ be the standard representation. We suppose that there is an epimorphism $f\colon G(K)\to S_4$. Then we may assume that 
$$
\tau\circ f(u)
=\tau(1234)=
\begin{pmatrix}
-1&1&0\\
-1&0&1\\
-1&0&0
\end{pmatrix}~\text{and}~
\tau\circ f(v)
=\tau(1243)=
\begin{pmatrix}
-1&1&0\\
-1&1&-1\\
0&1&-1
\end{pmatrix}.
$$
Putting 
$A=\tau\circ f(u)$ and $B=\tau\circ f(v)$, we obtain 
$$
P^{-1}AP
=
\begin{pmatrix}
1&1&0\\
0&0&1\\
0&1&0
\end{pmatrix},~
P^{-1}BP
=
\begin{pmatrix}
1&1&0\\
0&1&1\\
0&0&1
\end{pmatrix}
(\mathrm{mod}~ 2)~\text{for}~
P=
\begin{pmatrix}
1&0&0\\
0&1&0\\
1&0&1
\end{pmatrix}.
$$
If we define $\tau'\colon S_4\to \mathrm{GL}(2,\Z/2\Z)$ by 
$(1234)\mapsto {\tiny \begin{pmatrix}0&1\\1&0\end{pmatrix}}$, and 
$(1243)\mapsto {\tiny \begin{pmatrix}1&1\\0&1\end{pmatrix}}$, 
then $\tau'$ is conjugate to the mod $2$ reduction of the $2$-dimensional irreducible representation of $S_4$ which is lifted from $S_3$. 
Hence, the corresponding twisted Alexander polynomial is nonzero, because of the following argument. 

By the Fox derivative, $\Delta^{\tau' \circ f}_K(t) \equiv 0$ (mod $2$) is equivalent to the vanishing of
$$\det\Psi\left(\frac{\partial}{\partial v}(wuw^{-1}v^{-1})\right) = \det\Psi\left((1-v)\frac{\partial w}{\partial v} - 1\right).$$
Here, 
$\Psi$ is the homomorphism defined by 
$\Psi=(\tau'\circ f)\otimes\phi\colon \Z[F]\to \Z[S_4\times \Z]\to\mathrm{M}(2,\Z/2\Z)\otimes\Z[t^{\pm1}]$, 
where $F=\la u,v\ra$ is the free group of rank $2$ generated by $u,v$. 

If $\Psi\big(\frac{\partial w}{\partial v}\big) = 0$, the determinant above is equal to $1$ and hence $\Delta^{\tau' \circ f}_K(t) \not\equiv 0$ (mod $2$). If $\Psi\big(\frac{\partial w}{\partial v}\big) \neq 0$, $\Psi\big((1-v)\frac{\partial w}{\partial v}\big)$ has positive width with respect to $t$ and at least one of the highest and the lowest degrees is not zero; take a nonzero end degree $d$ and let $\xi \in {\rm M}(2,\mathbb{Z}/2\mathbb{Z})$ be the coefficient of $t^d$. Since
$$\tau'(A_4) = \left\{\begin{pmatrix} 1 & 0 \\ 0 & 1 \end{pmatrix}, \begin{pmatrix} 1 & 1 \\ 1 & 0 \end{pmatrix}, \begin{pmatrix} 0 & 1 \\ 1 & 1 \end{pmatrix} \right\}$$
for the alternating group $A_4$, any nonzero linear-sum of elements of $\tau'(A_4)$ is also contained in $\tau'(A_4)$. Therefore $\xi \in \tau'(A_4)$ if $d$ is even and $\xi \in {\tiny \begin{pmatrix} 0 & 1 \\ 1 & 0 \end{pmatrix}} \tau'(A_4)$ otherwise. In either case, $\xi$ has nonzero determinant and hence $\det\Psi\big((1-v)\frac{\partial w}{\partial v} - 1\big) \neq 0$.

Since the other $3$-dimensional representation is the tensor product of  alternating and standard representations, the corresponding twisted Alexander polynomial is also nonzero. This completes the proof of Theorem~\ref{thm:S_4}. 
\end{proof}

To conclude this section, we pose the following question:

\begin{question}
What is the minimum of $\ord(K)$ for the $2$-bridge knots $K$?
\end{question}

\section{TAV groups of order $pqr$}\label{section:pqr}
In this section, we consider the TAV groups whose orders are the products of at most three primes and especially give a concrete example of a knot admitting a TAV group $G$ of order $pqr$, where $p, q, r$ are distinct primes.

Let us construct some TAV groups. Let $p, q, r$ be distinct primes with $p < q < r$. We assume $q \equiv r \equiv 1 \;(\text{mod $p$})$ and take integers $a$ and $b$ that define elements of order $p$ in the multiplicative groups $(\mathbb{F}_q)^\times$ and $(\mathbb{F}_r)^\times$, respectively: We assume $a^p \equiv 1 \;(\text{mod $q$})$, $a \not\equiv 1 \;(\text{mod $q$})$, $b^p \equiv 1 \;(\text{mod $r$})$, and $b \not\equiv 1 \;(\text{mod $r$})$. The $a$-th and the $b$-th powers define automorphisms on the cyclic groups $C_q$ and $C_r$, and let $G(pqr;a,b)$ be the semidirect product $(C_q \times C_r) \rtimes C_p$ with respect to these automorphisms, i.e.,
$$G(pqr;a,b) = \langle x, y, z \mid x^q, y^r, z^p, xyx^{-1}y^{-1}, zxz^{-1}x^{-a}, zyz^{-1}y^{-b} \rangle.$$
We can easily check that $G(pqr;a, b)$ is a TAV group, because its commutator subgroup is $C_q\times C_r\cong C_{qr}$.

\begin{proposition}
Let $n$ be the product of at most three primes.
\begin{itemize}
\item[(i)] If $n$ has at most two prime divisors, there does not exist a TAV group of order $n$.
\item[(ii)] Let $n$ be the product $pqr$ of distinct primes $p, q, r$ with $p < q < r$. There exists a TAV group of order $n$ if and only if $q \equiv r \equiv 1 \;(\text{mod $p$})$. In this case, there exist exactly $p-1$ TAV groups
$$G(pqr; a, b^i) \qquad (i = 1, \dots, p-1)$$
of order $n$, where $a$ and $b$ are integers that define elements of order $p$ in the multiplicative groups $(\mathbb{F}_q)^\times$ and $(\mathbb{F}_r)^\times$, respectively.
\end{itemize}
\end{proposition}
\begin{proof}
(i) Let $G$ be a group of order $n$. If $n$ is a power of a prime $p$, then the commutator subgroup $G'$ of $G$ is also a $p$-group. If $n$ is the product $pq$ of two distinct primes $p$ and $q$, then $G$ is soluble, which means that the order of $G'$ is $1, p$, or $q$. Thus, $G$ is not a TAV group in these cases.

If $n = pq^2$ for distinct primes $p$ and $q$, it is known (e.g., \cite[Theorem 1.31]{Isaacs08-1}) that a group $G$ of order $n$ has a normal Sylow $p$-subgroup or a normal Sylow $q$-subgroup. Let $N$ be a normal Sylow subgroup. Since $|G / N|$ equals $q^2$ or $p$, $G/N$ is abelian and hence $G' \subset N$. Thus, $G'$ has prime-power order and $G$ is not a TAV group.

\noindent (ii) If $q \equiv r \equiv 1 \;(\text{mod $p$})$, there exists a TAV group $G(pqr;a,b)$ of order $pqr$ as seen above. To see the converse, recall that any group with square-free order is metacyclic (e.g., see \cite[Theorem 5.16]{Isaacs08-1}). Therefore, if $G$ is a TAV group of order $pqr$, $G/G' \cong C_{p'}$ and $G' \cong C_{q' r'} \cong C_{q'} \times C_{r'}$, where $\{ p, q, r \} = \{p', q', r' \}$. Since the actions of $C_{p'}$ on $C_{q'}$ and $C_{r'}$ are nontrivial, 
we have $q' \equiv r' \equiv 1 \;(\text{mod $p'$})$; in particular, $p' = p$ and this concludes the converse. Furthermore, this proof also shows that $G \cong G(pqr; a, b)$ for some $a$ and $b$.

Let us assume $q \equiv r \equiv 1 \;(\text{mod $p$})$. For $i = 1, \dots, p-1$ and any integers $a$ and $b$ having order $p$ in $(\mathbb{F}_q)^\times$ and $(\mathbb{F}_r)^\times$, we can define an isomorphism from $G(pqr; a^i, b^i)$ to $G(pqr; a, b)$ by
$$x \mapsto x, \quad y \mapsto y, \quad z \mapsto z^i.$$
Therefore any TAV group of order $pqr$ is isomorphic to one of
$$G(pqr; a, b), G(pqr; a, b^2), \dots, G(pqr; a, b^{p-1})$$
for fixed $a, b$.

Finally, we show that the $p - 1$ TAV groups $G(pqr; a, b^i)$ are not isomorphic. Let $\varphi\colon G(pqr; a, b^i) \to G(pqr; a, b^j)$ be an isomorphism. For any $\alpha \in (\mathbb{F}_q)^\times$ and $\beta \in (\mathbb{F}_r)^\times$,
$$x \mapsto x^\alpha, \quad y \mapsto y^\beta, \quad z \mapsto z$$
define an automorphism on $G(pqr; a, b^j)$. Therefore, by composing such an automorphism with $\varphi$, we may assume $\varphi (x) = x, \varphi(y) = y.$ Taking a conjugation by some element of $(G(pqr; a, b^j))'$, we can further assume that $\varphi(z) = z^k$ for some $k$. Since $\varphi(z) \varphi(x) \varphi(z)^{-1} \varphi(x)^{-a} = x^{a^k - a}$ and $\varphi(z) \varphi(y) \varphi(z)^{-1} \varphi(y)^{-b^i} = y^{b^{jk} - b^i}$, we have $k \equiv 1 \; (\text{mod $p$})$ and $jk \equiv i \;(\text{mod $p$})$, which concludes that $i = j$.
\end{proof}

\begin{theorem}[Theorem \ref{thm:main-6}]\label{thm:pqr}
Let $G$ be a TAV group of order $pqr$, where $p,q,r$ are primes.
Then there exist infinitely many hyperbolic knots that admit $G$ as a smallest TAV group;
in particular, there exist infinitely many hyperbolic knots $K$ with $\ord(K) = pqr$.
\end{theorem}



\begin{figure}[t]
\centering
\includegraphics[width=8cm]{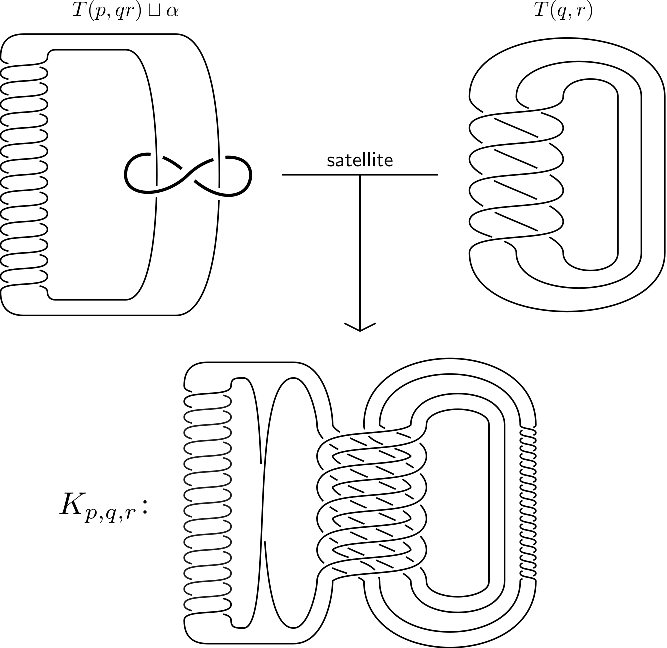}
\caption{Construction of $K_{p,q,r}$; $(p,q,r) = (2,3,5)$.}\label{fig:satellite}
\end{figure}

For a TAV group $G$ of order $pqr$, we construct a satellite knot as in Subsection \ref{subsec:3.0}: Let us take the torus knots $T(p, qr)$ and $T(q, r)$ as $K$ and $J$, respectively. Let $f_0 \colon G(K) \to G$ be a surjective homomorphism. We take an unknotted loop $\alpha$ disjoint from $K$ such that the linking number ${\rm lk}(K, \alpha)$ is zero and $f_0$ sends $\alpha$ to a generator of the commutator subgroup $C_{qr}$ of $G$. We denote the resulting satellite knot $K(\alpha, J)$ by $K_{p,q,r}$; see Figure \ref{fig:satellite} for an example. Let $f \colon G(K_{p,q,r}) \to G$ be the composite $G(K_{p,q,r}) \xrightarrow{\psi} G(K) \xrightarrow{f_0} G$.

\begin{proposition}\label{prop:pqr}
We have $\Delta_{K_{p,q,r}}^{\rho \circ f}(t) = 0;$ in particular, $G$ is a TAV group of $K_{p,q,r}$.
\end{proposition}
\begin{proof}
Since $G'$ is a cyclic group of order $qr$, we can apply Proposition \ref{const-prop-1} to $\Delta_{K_{p,q,r}}^{\rho \circ f}(t)$: Recalling that the Alexander polynomial $\Delta_J(t)$ of $J = T(q,r)$ is the $qr$-th cyclotomic polynomial, we have $\Delta_J(e^{2 \pi i/qr}) = 0$ and hence $\Delta_{K_{p,q,r}}^{\rho \circ f}(t) = 0.$
\end{proof}

In order to prove Theorem \ref{thm:pqr}, we prepare a purely algebraic lemma:

\begin{lemma}\label{lem:qr-grp}
Let $q, r$ be distinct primes with $q < r$ and let $H$ be a finite group that is generated by $x \in H$ with order $q$ and $y \in H$ with order $r$. If $|H| \leq q^2r/2$, $H$ is a metacyclic group of order $qr$.
\end{lemma}
\begin{proof}
Because $H$ has elements with orders $q$ and $r$, the order of $H$ is a multiple of $qr$; define $n = |H|/qr$. Let $Q, R$ be the subgroup generated by $x, y$, respectively. Since $n < q < r$, $Q$ and $R$ are Sylow subgroups.

We claim that at least one of $Q$ or $R$ is normal in $H$. Let $n_q, n_r$ be the numbers of the Sylow $q$-, $r$-subgroups of $H$, respectively, and suppose that $n_q, n_r >1$. By the Sylow theorem, $n_q \in q\mathbb{Z} + 1$ and especially $n_q > q$. Considering the transitive action of $H$ on the set of the Sylow $q$-subgroups, we find that $n_q$ divides $nr$; this implies that $r$ divides $n_q$ since $n_q > q > n$. Then, $n'_q := n_q/r$ is a positive integer less than or equal to $n$. In the same way, $n'_r := n_r/q$ is an integer with $1 \leq n'_r \leq n$.

Because $q$ and $r$ are coprime, we can take integers $a, b$ such that $ar + bq = qr + 1$, and we may further assume that $0 < a < q$ and hence $0 < b < r$. Here, the integer $a$ can be characterized as the unique integer such that $ar \in q \mathbb{Z} + 1$ and $0 < a < q$. Thus, we find $a = n'_q$ and similarly $b = n'_r$. We have $n_q + n_r = qr + 1$ and hence one of $n_q$ and $n_r$ is greater than $qr/2$. In either case, this means $|H| > qr/2 \cdot q = q^2r/2$, but this is a contradiction. Therefore $n_q = 1$ or $n_r = 1$.

Since $Q$ or $R$ is normal in $H$, we find $H = QR$. Recalling that $Q$ and $R$ are cyclic, we conclude that $H$ is a metacyclic group of order $qr$.
\end{proof}

\begin{proof}[Proof of Theorem \ref{thm:pqr}]
The existence of infinitely many hyperbolic knots admitting $G$ as a TAV group is a direct consequence of Proposition \ref{prop:pqr} and Theorem \ref{thm:hyperbolic}. After verifying $\ord(K_{p,q,r}) = pqr$, Theorem \ref{thm:hyperbolic} also shows that the family of hyperbolic knots contains infinitely many knots $K'$ with $\ord(K') = pqr$.

We first show that the image $H_J$ of any homomorphism $g_J \colon G(J) \to H$ from $G(J)$ to any group $H$ with $|H| < pqr$ is metabelian. Let us recall that $G(J) = G(T(q, r))$ has a well-known presentation $\langle x, y \mid x^q = y^r \rangle$ and that its center is the infinite cyclic group generated by $x^q = y^r$. Thus, $H_J/Z(H_J)$, where $Z(H_J)$ is the center of $H_J$, is generated by two elements $\overline{g_J(x)}$ and $\overline{g_J(y)}$, where an over-line expresses the image in the quotient group $H_J/Z(H_J)$, and we find $\overline{g_J(x)}^q = \overline{g_J(y)}^r = \overline{e}$. If $\overline{g_J(x)}$ or $\overline{g_J(y)}$ is the identity, $H_J$ is generated by a single element and hence cyclic. Otherwise, we can apply Lemma \ref{lem:qr-grp} (we should recall that $p \leq (q-1)/2$) to $H_J/Z(H_J)$ to find that $H_J/Z(H_J)$ is a metacyclic group of order $qr$. Because the second homology group of a metacyclic group of order $qr$ is trivial, any central extension preserves the commutator subgroup and only extends the abelianization. Thus, $H_J$ is metabelian. Furthermore, this proof shows that if $H_J$ is not cyclic, then $|H|$ is a multiple of $qr$.

Let $g\colon G(K_{p, q, r}) \to H$ be a surjective homomorphism onto a group $H$ with $|H| < pqr$. Since $E_{K_{p,q,r}} = E_{K \cup \alpha} \cup E_J$, we have the restriction homomorphism $g_J \colon G(J) \to H$ and its image $H_J$ is metabelian as seen above; in particular, the second derived group $(H_J)''$ is trivial. Recalling that a preferred longitude of $J$ and a meridian of $\alpha$ are glued together in $E_{K_{p,q,r}}$, we find that $g$ induces a group homomorphism $g_K \colon G(K) \to H$.

If $|H| \in qr \mathbb{Z}$, $p$ does not divide $|H|$ and hence the image of $g_K\colon G(K) = G(T(p, qr)) \to H$ is cyclic. In this case, (a preferred longitude of) $\alpha$ is mapped to $e \in H$. Since a preferred longitude of $\alpha$ is identified with a meridian of $J$, $g_J$ is trivial and it follows that $H$ is cyclic; especially, $H$ is not a TAV group.

If $|H| \not\in qr \mathbb{Z}$, $H_J$ is cyclic as seen above, which also implies that $g \colon G(K_{p,q,r}) \to H$ factors through $\psi \colon G(K_{p,q,r}) \to G(K)$, i.e., there exists a homomorphism $g_0\colon G(K) \to H$ such that $g = g_0 \circ \psi$. Since $K = T(p,qr)$ is fibered, $\Delta_K^{\rho \circ g_0}(t) \neq 0$. Also, because the roots of the Alexander polynomial $\Delta_J(t)$ of $J = T(q,r)$ are the primitive $qr$-th roots of unity, the assumption $|H| \not\in qr\mathbb{Z}$ implies that $\Delta_J(e^{2 \pi i/|H_J|}) \neq 0$. Thus, Proposition \ref{const-prop-1} concludes that $\Delta_{K_{p,q,r}}^{\rho \circ g}(t) \neq 0$: $H$ is not a TAV group of $K_{p,q,r}$.
\end{proof}

\begin{corollary}\label{cor:273}
There exist infinitely many hyperbolic knots that admit distinct smallest TAV groups. Moreover, there exist infinitely many TAV orders $\ord(K)$ which correspond to such smallest TAV groups. 
\end{corollary}

\begin{proof}
Let $G_i~(i=1,2)$ be the finite groups defined by
\begin{align*}
G_1&=C_{91}\rtimes C_3=\la x,y\,|\,x^{91}, y^3,yxy^{-1}x^{-74}\ra, \\
G_2&=C_{91}\rtimes_4 C_3=\la x,y\,|\, x^{91},y^3,yxy^{-1}x^{-9}\ra.
\end{align*} 
They appear in GroupNames \cite{GN} as $(273,3)$ and $(273,4)$
respectively, and share the order $273=3\cdot7\cdot13$. 

If we apply Theorem \ref{thm:pqr} to these TAV groups $G_1$ and $G_2$, we can obtain infinitely many hyperbolic knots as desired. Because the proof of Theorem \ref{thm:pqr} depends only on the order of these TAV groups. 
Since it is well known that there exist infinitely many prime numbers $q$ satisfying $q\equiv1~(\mathrm{mod} ~3)$, the latter assertion follows immediately.
\end{proof}

In our previous paper \cite{IMS23-1}, we proposed the following problem:

\begin{problem}[{\cite[Problem~5.3]{IMS23-1}}]\label{que:5.3}
Find a non-fibered knot $K$ and an epimorphism $f \colon G(K) \to D_{15}$ such that $\D_K^{\rho\circ f}(t)=0$. 
\end{problem}

As a corollary of Theorem \ref{thm:pqr}, we can give an answer to Problem \ref{que:5.3}. 

\begin{corollary}\label{cor:D_15}
Let $K=K_{2,3,5}$ be the knot introduced after Theorem \ref{thm:pqr}. 
Then, the following hold:
\begin{itemize}
\item[(i)]
There is an epimorphism $f\colon G(K)\to D_{15}$ such that 
$\D_K^{\rho\circ f}(t)=0$, 
\item[(ii)]
For any epimorphism $f\colon G(K) \to G$ with $|G|<30$, 
$\D_K^{\rho\circ f}(t)\not=0$. 
\end{itemize}
In particular, it holds that $\ord(K)=30$. 
\end{corollary}

\section{Twisted Alexander polynomials for a central extension}\label{sec:extension}

In this section, we provide a formula for the twisted Alexander polynomial of a central extension of a finite group of weight one by a cyclic group. Using the formula, we can find an important fact about the image of the TAV order function $\ord|_{\mathcal{N}}\colon\mathcal{N}\to\N$ (see Example \ref{ex:dic33} for details).  

Let $G_{k,1}$ be a finite group with the abelianization $C_k=\la \bar{x}\,|\,\bar{x}^k\ra$, namely, it is of weight one and fits the following short exact sequence: 
$$
1\longrightarrow H \longrightarrow G_{k,1} \overset{\pi}{\longrightarrow} C_{k} \longrightarrow 1,
$$
where $H$ is the commutator subgroup of $G_{k,1}$. We then define a group $G_{k,n}$ as the pull-back of two epimorphisms $\pi\colon G_{k,1}\to C_k$ and $\pi_k\colon C_{kn}=\la x\,|\,x^{kn}\ra\to C_k$, that is, $G_{k,n}=\{(z,x^j)\in G_{k,1}\times C_{kn}\,|\,\pi(z)=\pi_k(x^j)\}$:
\begin{equation*}
\begin{CD}
1@>>> H @>>> G_{k,n} @>{\pr_2}>> C_{kn} @>>> 1\\
@. @VV{\mathrm{id}_H}V @VV{\pr_1=\,\pr}V @VV{\pi_k}V \\
1@>>> H @>>> G_{k,1} @>{\pi}>> C_{k} @>>> 1
\end{CD}
\end{equation*}
where $\pr_\mu$ is the projection onto the $\mu$-th  component. We note that $\pr_2\colon G_{k,n}\to C_{kn}$ is the abelianization, and the inclusion $\iota\colon C_n=\la y\,|\,y^n\ra\to G_{k,n},\,y\mapsto (e,x^k)$ is a homomorphism such that 
$$
1\longrightarrow C_n \overset{\iota}{\longrightarrow} G_{k,n}\overset{\pr}{\longrightarrow} G_{k,1}\longrightarrow 1
$$ 
is a central extension. We call the group $G_{k,n}$ the \textit{$n$-th central extension} of $G_{k,1}$. If the group $G_{k,1}$ does not appear as the $n$-th central extension $(n\geq2)$ of any other group, then we call it a \textit{seed} of any series of central extensions. At the end of this section, we discuss a characterization of seeds. 

Now, we prepare a lemma. Let $K$ be a knot in $S^3$ (not necessarily non-fibered). 

\begin{lemma}\label{lem:lifting}
For a given epimorphism $f_1\colon G(K) \to G_{k,1}$, there exists an epimorphism $f_n\colon G(K) \to G_{k,n}$ such that $\pr\circ f_n=f_1$. 
\end{lemma}

\begin{proof}
Let $\bar{\phi}\colon G(K)\to C_{kn}$ be the composition of the abelianization $\phi\colon G(K)\to\Z$ and the projection $\Z\to\Z/kn\Z\cong C_{kn}$. 
We define a lift $f_n\colon G(K)\to G_{k,n}$ to be $f_n(u)=(f_1(u),\bar{\phi}(u))$ for $u\in G(K)$. It is clearly a homomorphism and satisfies $\pr\circ f_n=f_1$. Since $f_1(G(K)') = H$, we have $f_n(G(K)') = H \times \{e\} = G_{k,n}'$.
Also, since $\bar{\phi}$ is an epimorphism, there is an element $u \in
G(K)$ such that $\bar{\phi}(u) = {\rm pr}_2 \circ f_n(u) \in C_{kn}$ is a
generator of $C_{kn}$. Because ${\rm pr}_2 \colon G_{k,n} \to C_{kn}$ is
the abelianization, $G_{k,n}$ is generated by $G_{k,n}' = f_n(G(K)')$ and
$f_n(u)$, which means that $f_n$ is surjective.
\end{proof}

Let $K$ be a knot and assume that there exists an epimorphism $f_n\colon G(K) \to G_{k,n}$. Then we have the following formula for the twisted Alexander polynomial associated with the regular representation $\rho_n\colon G_{k,n}\to \mathrm{GL}(kn|H|,\C)$. For simplicity, let us denote by $\tilde{\rho}_1$ the composition of $\pr\colon G_{k,n}\to G_{k,1}$ and the regular representation $\rho_1\colon G_{k,1}\to \mathrm{GL}(k|H|,\C)$. 

\begin{theorem}[Theorem \ref{thm:main-9}]\label{thm:semi-direct}
$\displaystyle{
\D_K^{\rho_n\circ f_n}(t)
=
\prod_{l=0}^{n-1}
\D_K^{\tilde{\rho}_1 \circ f_n}\big(e^{\frac{2l\pi i}{kn}}t\big).}$
\end{theorem}

Namely, the twisted Alexander polynomial of the $n$-th central extension $G_{k,n}$ is determined by that of $G_{k,1}$. Note that the above formula also holds when $G_{k,n}$ is not a TAV group. Using Proposition \ref{pro:quotient}, Lemma \ref{lem:lifting} and Theorem \ref{thm:semi-direct}, we can obtain the following corollary:

\begin{corollary}\label{cor:smallest}
The following two conditions are equivalent:
\begin{itemize}
\item[(i)]
There exists an epimorphism $f_1\colon G(K) \to G_{k,1}$ such that $\D_K^{\rho_1\circ f_1}(t)=0$.
\item[(ii)]
There exists an epimorphism $f_n\colon G(K) \to G_{k,n}$ such that $\D_K^{\rho_n\circ f_n}(t)=0$.
\end{itemize}
In particular, the set $\mathcal{V}(G_{k,1})$ coincides with $\mathcal{V}(G_{k,n})$ for any $n\geq 2$. Accordingly, if $G_{k,n}~(n\geq2)$ is a TAV group, then it is not the smallest for any non-fibered knots. 
\end{corollary}

Before proving Theorem \ref{thm:semi-direct}, let us see some examples. 

\begin{example}\label{ex:extension}
(i) Let $H=C_q$ be the cyclic group of odd order $q$ and $k=2$. Then the group $C_q\rtimes C_{2n}$ is a central extension of the dihedral group $D_q=C_q\rtimes C _2$:
$$
1\longrightarrow C_n \longrightarrow C_q\rtimes C_{2n} \overset{\pr}{\longrightarrow} D_q \longrightarrow 1.
$$
In particular, when $n=2$, $C_q\rtimes C_4$ is the dicyclic group $\mathrm{Dic}_q$ appeared in Section \ref{sec:2}. Using Theorem \ref{thm:semi-direct}, we have 
$$
\D_K^{\rho_2\circ f_2}(t)
=
\D_K^{\tilde{\rho}_1 \circ f_2}(t)\cdot
\D_K^{\tilde{\rho}_1 \circ f_2}\left(it\right).
$$
Similarly, when $n=3$, 
$$
\D_K^{\rho_3\circ f_3}(t)
=
\D_K^{\tilde{\rho}_1 \circ f_3}(t)\cdot
\D_K^{\tilde{\rho}_1 \circ f_3}\big(e^{\frac{\pi i}{3}}t\big)\cdot
\D_K^{\tilde{\rho}_1 \circ f_3}\big(e^{\frac{2\pi i}{3}}t\big)
$$
holds. Hence, Corollary \ref{cor:smallest} implies that 
$\mathcal{V}(D_q)=\mathcal{V}(\mathrm{Dic}_q)=\mathcal{V}(C_q\rtimes C_6)$ hold.  

(ii) Let $H$ be the alternating group $A_4$. Then we have a central extension 
of the symmetric group $S_4=A_4\rtimes C_2$:
$$
1\longrightarrow C_n \longrightarrow A_4\rtimes C_{2n} \overset{\pr}{\longrightarrow} S_4 \longrightarrow 1.
$$
Thus, Corollary \ref{cor:smallest} implies that 
$\mathcal{V}(S_4)=\mathcal{V}(A_4\rtimes C_{2n})$ holds for $n\geq 2$. 

(iii) Let $H$ be $\mathrm{SL}_2(\F_3)$, and $G_{k,1}$ a non-split extension of $C_2$ by $\mathrm{SL}_2(\F_3)$; i.e. $G_{k,1}=\mathrm{SL}_2(\F_3).C_2$. Then we have a central extension $G_{k,n}=\mathrm{SL}_2(\F_3).C_{2n}$ of $G_{k,1}$: 
$$
1\longrightarrow C_{n} \longrightarrow G_{k,n} \overset{\pr}{\longrightarrow} G_{k,1} \longrightarrow 1.
$$
Hence, Corollary \ref{cor:smallest} implies that $\mathcal{V}(G_{k,1})=\mathcal{V}(G_{k,n})$ holds for $n\geq2$. 
\end{example}

Let $\eta_l\colon C_{kn}\to \mathrm{GL}(1,\C)$ be a one-dimensional representation defined by $\eta_l(x)=\exp\big(\frac{2l\pi i}{kn}\big)$, and $\tau_l\colon G_{k,n}\to \mathrm{GL}(1,\C)$ the composition of $\pr_2\colon G_{k,n}\to C_{kn}$ and $\eta_l$. We then define a representation $\tau_{n,l}\colon G_{k,n}\to \mathrm{GL}(k|H|,\C)$ by $\tau_{n,l}=\tau_l\otimes \tilde{\rho}_1$. It is easy to see that 
$$
\tau_{n,l}(z,x^j)=(\tau_l\otimes\tilde{\rho}_1)(z,x^j)
=\tau_l(z,x^j)\otimes\tilde{\rho}_1(z,x^j)
=\eta_l(x^j)\otimes\rho_1(z)
$$
hold for $(z,x^j)\in G_{k,n}$. 

\begin{proposition}\label{pro:regular2}
The representation $\bigoplus_{l=0}^{n-1}\tau_{n,l}$ is conjugate to 
the regular representation $\rho_{n}\colon G_{k,n}\to \mathrm{GL}(kn|H|,\C)$. 
\end{proposition}

\begin{proof}
Note that for a finite group $G$, a given representation $\rho\colon G\to\mathrm{GL}(|G|,\C)$ is conjugate to the regular representation of $G$ if and only if the character $\chi_\rho$ satisfies $\chi_\rho(x)=|G|$, if $x=e$, and $\chi_\rho(x)=0$, otherwise. 
For $(z,x^j)\in G_{k,n}~(0\leq j\leq kn-1)$, we have 
\begin{align*}
\chi_{\bigoplus_{l=0}^{n-1}\tau_{n,l}}(z,x^j)
&=
\sum_{l=0}^{n-1}\chi_{\tau_{n,l}}(z,x^j)
=
\sum_{l=0}^{n-1}\mathrm{tr}\big(\tau_{n,l}(z,x^j)\big)\\
&=
\sum_{l=0}^{n-1}\mathrm{tr}\big(\eta_l(x^j)\otimes\rho_1(z)\big)\\
&=
\sum_{l=0}^{n-1}\mathrm{tr}\big(\large(e^{\frac{2l\pi i}{kn}}\large)^j\cdot \rho_1(z)\big)
=
\sum_{l=0}^{n-1}
e^{\frac{2jl\pi i}{kn}}\mathrm{tr}\large(\rho_1(z)\large)\\
&=
\begin{cases}
\displaystyle{
\sum_{l=0}^{n-1}e^{\frac{2\bar{j}l\pi i}{n}}k|H|
}
, & \text{if $z=e$ and $j=k\bar{j}~(0\leq \bar{j}\leq n-1)$}\\
0, & \text{otherwise}
\end{cases}\\
&=
\begin{cases}
kn|H|, & \text{if $z=e$ and $j=0$ by Lemma \ref{lem:sum} below}\\
0, & \text{otherwise}.
\end{cases}
\end{align*}
This completes the proof of Proposition \ref{pro:regular2}.
\end{proof}

\begin{lemma}\label{lem:sum}
$\displaystyle{
\sum_{l=0}^{n-1}e^{\frac{2dl\pi i}{n}}
=
\begin{cases}
n,& d=0\\
0,& 0<d<n
\end{cases}
}$
\end{lemma}

\begin{proof}
The case $d=0$ is clear. When $0<d<n$, we have $\big(\large(e^{\frac{2l\pi i}{n}}\large)^d\big)^n=e^{2dl\pi i}=1$ and $\large(e^{\frac{2l\pi i}{n}}\large)^d\not=1$. 
In general, suppose that $(\eta^d)^n=1$ and $\eta^d\not=1$, then 
we have $\sum_{l=0}^{n-1}(\eta^d)^l=0$ because $(\eta^d-1)\big((\eta^d)^{n-1}+(\eta^d)^{n-2}+\cdots+\eta^d+1\big)=0$ holds.
\end{proof}

Let us provide a proof of Theorem \ref{thm:semi-direct}. 
By Proposition \ref{pro:regular2}, we see that the composition $\rho_n\circ f_n$ is conjugate to 
\begin{align*}
\big(
\bigoplus_{l=0}^{n-1}
\tau_{n,l}
\big)\circ f_n
=
\big(
\bigoplus_{l=0}^{n-1}
\tau_l\otimes\tilde{\rho}_1
\big)\circ f_n
=
\bigoplus_{l=0}^{n-1}(\tau_l\circ f_n)\otimes
(\tilde{\rho}_1 \circ f_n).
\end{align*}
Hence, 
\begin{align*}
\D_K^{\rho_n\circ f_n}(t)
&=
\prod_{l=0}^{n-1}\D_K^{(\tau_l\circ f_n)\otimes(\tilde{\rho}_1 \circ f_n)}(t)
=
\prod_{l=0}^{n-1}\D_K^{\tilde{\rho}_1 \circ f_n}\big(
e^{\frac{2l\pi i}{kn}}t\big)
\end{align*}
hold as desired.

Finally we end this section with a characterization of seeds of a series of central extensions. 

\begin{proposition}\label{pro:seed}
Let $G$ be a finite group of weight one. Then $G$ is not a seed if and only if there exists a nontrivial cyclic group $C$ in the center $Z(G)$ such that $C\cap G'=\{e\}$. 
\end{proposition}

\begin{proof}
Assume that $G=G_{k,n}~(n\geq 2)$. Let $C_n$ be a subgroup of $G_{k,n}$ generated by $(e,x^k)$. Then $C_n\subset Z(G_{k,n})$ and $C_n\cap H=\{e\}$ hold, where $H$ is the commutator subgroup of $G_{k,n}$. 

Conversely, if there exists a nontrivial cyclic group $C$ in $Z(G)$ such that $C\cap G'=\{e\}$, we can define the group $G_{k,1}$ by $G/C$. Namely, $G$ is a central extension of $G_{k,1}$ by $C$. 
\end{proof}

\section{Concluding remark}\label{sec:5}

In this section, we discuss (I) the image of the TAV order function $\ord\colon\mathcal{K}\to\N$, (II) some numerical invariants of TAV groups, and (III) a relationship between the TAV order and the invariant coming from the circle valued Morse function. Moreover, we explain (IV) a question on the twisted Alexander polynomials of $3$-manifold groups associated with representations of finite groups. 

{\bf (I)} 
Using Theorem \ref{thm:main-4}, we know that the intersection of the image of the TAV order function $\ord|_{\mathcal{N}} \colon \mathcal{N}\to\mathbb{N}$ 
and the closed interval $[1,200]$ is contained in the finite set 
\begin{align*}
\{&24,30,42,48,60,66,70,72,78,84,90,96,102,110,114,120,126,130,132,\\
&138, 140,144,150,154,156,160,168,170,174,180,182,186,190,192,198\}
\end{align*}
consisting of the orders of TAV groups $G$ with $1\leq|G|\leq200$ 
(see Tables 1 to 3).  
However, it is not known whether these values are actually realized as the TAV orders of 
non-fibered knots. We then proposed the following problem in our previous paper \cite{IMS23-1}. 

\begin{problem}[{\cite[Problem~5.4]{IMS23-1}}]\label{problem:image}
Determine the image of the TAV order function $\ord|_{\mathcal{N}} \colon \mathcal{N}\to\mathbb{N}$. 
\end{problem}

We see from Theorems \ref{thm:main-2} and \ref{thm:main-6} that $\mathrm{Im}\, \ord$ contains at least the finite set 
\begin{align*}
\{&\underline{24},30,42,\underline{60},66,70,78,84,\underline{96},102,110,114,\underline{120},130,\\
&138, 154,\underline{168},170,174,182,186,190\}
\end{align*}
where the orders with underline (e.g., $\underline{24}, \underline{60}$) come from Theorem \ref{thm:main-2}. Here, we note that the order $84$ is realized by a knot which is constructed by using Theorem \ref{const-thm} (see the appendix for details). As for the concrete TAV groups of these orders in the above finite set, see Tables 1 to 3 in the appendix. 

As shown in Corollary \ref{cor:smallest}, there are several series of TAV groups $G_{k,n}$ such that $\mathcal{V}(G_{k,1})=\mathcal{V}(G_{k,n})$ holds for $n\geq2$. Thus, the TAV group $G_{k,n}~(n\geq2)$ is not the smallest for any non-fibered knots. Using this fact, we find the following important fact on $\ord(K)$. 

\begin{example}\label{ex:dic33}
Let us consider the dicyclic group $G=\mathrm{Dic}_{33}$. This is a central extension of the dihedral group $D_{33}$ and the unique TAV group of order $132$ (see Table 2). Hence, we can conclude that the order $132$ is \textit{not} contained in the image of the TAV order function $\ord|_{\mathcal{N}} \colon \mathcal{N}\to\mathbb{N}$, because $\mathcal{V}(\mathrm{Dic}_{33})=\mathcal{V}(D_{33})$ holds. 
\end{example}

{\bf (II)} 
As with the crossing number $c(G)$ and the bridge number $b(G)$ of a TAV group $G$, we can use numerical invariants of knots to introduce new invariants of $G$. For example, we can define the \textit{genus of a TAV group} $G$ to be the minimal genus $g(K)$ of knots $K$ contained in $\mathcal{V}(G)$; $g(G):=\min_{K\in\mathcal{V}(G)}\{g(K)\}$. 
Since every nontrivial knot has the minimal genus at least one, $g(G)\geq 1$ holds for any TAV group $G$. In view of Corollary \ref{cor:smallest}, 
for a central extension $G_{k,n}$ of $G_{k,1}$ by $C_n$, we see that $c(G_{k,n})=c(G_{k,1}),~b(G_{k,n})=b(G_{k,1})$, and $ g(G_{k,n})=g(G_{k,1})$ hold for $n\geq2$. 

\begin{question}
Determine $c(G),b(G)$ and $g(G)$ for a given TAV group $G$. In particular, 
is $g(G)$ one for any TAV group $G$?
\end{question}

We show the crossing number $c(G)$ for some TAV groups $G$ of order less than $201$ in the appendix of this paper. We plan to address this question systematically in future work. 

{\bf (III)} 
Given an oriented knot $K$ in $S^3$, the \textit{Morse-Novikov number} $\MN(K)$ of $K$ is the count of the minimum number of critical points among regular Morse functions $h\colon S^3\setminus K\to S^1$ (see \cite{PRW01-1}). It is known that $\MN(K)=0$ if and only if $K$ is fibered. Since the Euler characteristic of the complement of $K$ is zero, $\MN(K)$ is always even. Moreover, Baker shows in \cite{Baker21-1} that the Morse-Novikov number of knots is additive under connected sum and unchanged by cabling. 
To be precise, the following equalities hold: 
\begin{itemize}
\item
(\cite[Theorem 1.1]{Baker21-1})
If $K=K_1\#K_2$ is a connected sum of two oriented knots $K_1$ and $K_2$ in $S^3$, then $\MN(K)=\MN(K_1)+\MN(K_2)$.
\item
(\cite[Theorem 1.3]{Baker21-1})
Let $K_{p,q}$ be the $(p,q)$-cable of the knot $K$ for coprime integers $p,q$ with $p>0$. Then $\MN(K_{p,q})=\MN(K)$.
\end{itemize}

In a sense, the Morse-Novikov number of $K$ measures how far a given knot $K$ is from a fibered knot. Accordingly, it is natural to raise the following question: 
\textit{Is there any relationship between $\ord(K)$ and $\MN(K)$?}

\begin{example}
There is a knot $K$ with $\MN(K)=2$ and $\ord(K)=24$. In fact, we may take $K=9_{35}$ or $9_{46}$ as an example (see \cite{Goda06-1} and Theorem \ref{thm:main-2}(i)). This shows that $K$ is the closest to a fibered knot in the sence of the Morse-Novikov number, but it is the farthest from a fibered knot in view of the TAV order (see Theorem \ref{thm:main-1}(i) and the definition of $\ord(K)$ for a fibered knot). 

Moreover, using the results of Baker mentioned above, for any even integer $n=2m>0$, we can find a knot $K_n$ with $\MN(K_n)=n$ and $\ord(K_n)>n$. Namely, $K_n$ is arbitrary far from a fibered knot in the sense of the Morse-Novikov number, but $K_n$ approches a fibered knot arbitrarily in view of the TAV order. Actually, we only have to take $K_n$ as the $(n!,1)$-cable of the connected sum $\#_m K$ for $K=9_{35}$ or $9_{46}$ (see \cite[Proposition 2.3]{IMS23-1} and the proof of \cite[Theorem 1.2(ii)]{IMS23-1}). 

Unlike the Morse-Novikov number $\mathcal{MN}(K)$, the total order of $\mathrm{Im}\,\ord|_{\mathcal{N}}$ in $\N$ does not seem to behave well with respect to the hierarchy of non-fibered knots. 
\end{example}

{\bf (IV)} 
Finally, we repost a related question on representations of $3$-manifold groups from our previous paper \cite{IMS23-1}. 
For a compact, orientable, connected $3$-manifold 
$N$ with toroidal or empty boundary, if $\phi\in H^1(N;\Z)=\mathrm{Hom}(\pi_1(N),\Z)$ is 
a non-fibered class, 
then the \textit{twisted Alexander vanishing (TAV) order} $\ord(N,\phi)$ is defined 
to be the smallest order of a finite group $G$ such that 
there exists an epimorphism $f \colon \pi_1(N)\to G$ 
with $\D_{N,\phi}^{\rho\circ f}(t)=0$ (see also \cite{MS22-1}). 
Here, in the case of the complement of a $k$-component link in $S^3$, 
Gonz\'{a}lez-Acu\~{n}a shows that there is an epimorphism from the link group onto a group $G$ if and only if it is finitely generated and has weight at most $k$ (see \cite[Theorem 3]{GA75-1}). 

Since nothing is known for the TAV order $\ord(N,\phi)$ of $3$-manifold groups so far, we propose the following question for further study:

\begin{question}[{\cite[Problem~5.5]{IMS23-1}}]
Study the basic properties of $\ord(N,\phi)$ as described in Theorem~\ref{thm:main-1}. 
\end{question}

Recently, Hughes and Kielak introduced the notion of \textit{TAP groups} (standing for twisted Alexander polynomial), which are groups where the twisted Alexander polynomials associated with representations of finite groups control algebraic fibering (see \cite{HK25-1} for details). An epimorphism $\phi\colon G\to\Z$ is said to be \textit{algebraically fibered} if $\ker\phi$ is finitely generated. A group $G$ is \textit{algebraically fibered} if it admits such an epimorphism. Consequently, developing the theory of TAV groups and TAV orders within this framework appears to be both a challeging and interesting task. 

\section*{Appendix}

In this appendix, we provide a list of the TAV groups of order less than $201$. In fact, using Theorem \ref{thm:main-4} and Proposition \ref{pro:weight}(i), we can make such a list from GroupNames (see \cite{GN}), which is shown in Tables 1 to 3. 

By computer-aided calculation, we can determine the prime knots $K$ up to $10$ crossings which admit a TAV group $G$ of order less than $201$. Using Theorem \ref{thm:pqr} and Corollary \ref{cor:smallest}, we see that for dihedral groups $G_{2,1}=D_{qr}$ and their central extensions $G_{2,n}~(n\geq2)$, the sets $\mathcal{V}(G_{2,n})~(n\geq1)$ contain the knot $K_{2,q,r}$ defined in Section \ref{section:pqr}. In particular, knots written in red are placed in the row associated with their respective smallest TAV groups. Using Theorem \ref{const-thm}, we can construct a knot $K(\alpha,J)$ for a given TAV group, which implies that we can fill in the all blanks in Tables 1 to 3. However, we can say nothing about the minimality of the TAV group, so that we leave some rows blank. 

The knot $6_1^{(3,5)}$ appeared in Tables 1 and 3 is constructed as follows: First, we take the knot $K=6_1$ which admits an epimorphism $f\colon G(K) \to C_3\rtimes F_5$. Applying the similar satellite knot construction for $K$ and $J=T(3,5)$ as in Subsection \ref{subsec:3.0}, we obtain the knot $6_1^{(3,5)}=K(\alpha,J)$. Theorem \ref{const-thm} asserts $6_1^{(3,5)}\in\mathcal{V}(C_3\rtimes F_5)$,  and by computer-aided calculation, 
we see that $\D_{6_1^{(3,5)}}^{\rho\circ f}(t)\not=0$ for the TAV groups of order less than $60$. Hence, it follows that $\ord(6_1^{(3,5)})=60$. Moreover, we can also check that $6_1^{(3,5)}\in\mathcal{V}(A_5)$. Therefore the knot $6_1^{(3,5)}$ admits distinct smallest TAV groups (see Corollary \ref{cor:273}). The knot $8_{11}^{(2,7)}$ appeared in Tables 1 and 2 is constructed in the same way. 

The presentation $(G;G',G/G'\colon\#\text{GroupNames})_{*}$ with a subscript $*$ in Tables 1 to 3 implies that $G$ is a member of the series $*$ of central extensions as in Example \ref{ex:extension}. We see that there are 19 series of central extensions of TAV groups of order less than $201$. Here we simply present a TAV group $G$ by $(|G|,\#\text{GroupNames})$. Note that there are some overlaps among them, which are denoted by underlines (see the items d.\,e.\,and r.\,below). Namely, those groups admit distinct seeds.   
\begin{enumerate}
\item[a.]
$A_4 \rtimes C_{2n}~ (n=1,S_4)$: \\
(24,12), (48,30), (72,42), (96,65), (120,37), (144,123), (168,45), (192,186)
\item[b.]
$C_{15} \rtimes C_{2n}~ (n=1,D_{15}, n=2,\mathrm{Dic}_{15})$: \\
(30,3), (60,3), (90,7), (120,3), (150,11), (180,15)
\item[c.]
$C_{21} \rtimes C_{2n}~ (n=1,D_{21}, n=2,\mathrm{Dic}_{21})$: 
(42,5), (84,5), (126,13), (168,5)
\item[d.]
$\mathrm{SL}_2(\F_3).C_{2n}~(n=1,\mathrm{CSU}_2(\F_3))$: 
(48,28), \underline{(96,66)}, (144,121), \underline{(192,183)}
\item[e.]
$\mathrm{SL}_2(\mathbb{F}_3) \rtimes C_{2n}$: 
(48,29), \underline{(96,66)}, (144,122), \underline{(192,183)}
\item[f.]
$A_5 \times C_n$: 
(60,5), (120,35), (180,19)
\item[g.]
$C_{15} \rtimes C_{4n}$: 
(60,7), (120,7), (180,21)
\item[h.]
$C_{33} \rtimes C_{2n}~ (n=1,D_{33}, n=2,\mathrm{Dic}_{33})$: 
(66,3), (132,3), (198,7)
\item[i.]
$C_{35} \rtimes C_{2n}~ (n=1,D_{35}, n=2,\mathrm{Dic}_{35})$: 
(70,3), (140,3)
\item[j.]
$C_3.A_4 \rtimes C_{2n}~ (n=1,C_3.S_4)$: 
(72,15), (144,33)
\item[k.]
$(C_3 \times A_4) \rtimes C_{2n}~ (n=1,C_3 \rtimes S_4)$: 
(72,43), (144,126)
\item[l.]
$(C_2 \times C_6) \rtimes C_{6n}~ (n=1,S_3 \times A_4)$: 
(72,44), (144,129)
\item[m.]
$C_{39} \rtimes C_{2n}~ (n=1,D_{39}, n=2,\mathrm{Dic}_{39})$: 
(78,5), (156,5)
\item[n.]
$(C_2 \times C_{14}) \rtimes C_{3n}$: 
(84,11), (168,53)
\item[o.]
$C_{45} \rtimes C_{2n}~ (n=1,D_{45}, n=2,\mathrm{Dic}_{45})$: 
(90,3), (180,3)
\item[p.]
$(C_3 \times C_{15}) \rtimes C_{2n}~ (n=1,C_3 \rtimes D_{15})$: 
(90,9), (180,17)
\item[q.]
$(C_4^2 \rtimes C_3) \rtimes C_{2n}~ (n=1,C_4^2 \rtimes S_3)$: 
(96,64), (192,182)
\item[r.]
$\mathrm{SL}_2(\mathbb{F}_3) \rtimes C_{4n}$: 
(96,67), \underline{(192,183)}
\item[s.]
$(C_2^2 \rtimes A_4) \rtimes C_{2n}~ (n=1,C_2^2 \rtimes S_4)$: 
(96,227), (192,1495)
\end{enumerate}

As mentioned in Example \ref{ex:extension}(iii), the group $\mathrm{SL}_2(\F_3).C_{2n}$ is a non-split extension of $C_{2n}$ by $\mathrm{SL}_2(\F_3)$. In Tables 1 to 3, there are several different groups that are extensions of a given group $Q$ by a given group $N$, in which case they are ordered, and the index is shown as a subscript of ``\,.\,'' or ``$\rtimes$''; namely, $N._kQ$ or $N\rtimes_kQ$ for the $k$-th extension. 

Finally, the crossing number $c=c(G)$ of some TAV groups $G$ of order less than $201$ are shown in the most right column in each table. 

A systematic treatment of TAV groups of order greater than 200 will be presented in our forthcoming paper.


\begin{table}\label{table:3}
\begin{center}
\begin{tabular}{|c|l|l|c|}
\hline
$|G|$ & $(G;G',G/G'\colon\#\, \text{GroupNames})$ & $K$ & $c$ \\
\hline
24 & $(S_4;A_4, C_2\colon 12)_{\text{a}}$ & \color{red}{$9_{35},9_{46}$} & $9$ \\
\hline
30 & $(D_{15};C_{15},C_2\colon3)_{\text{b}}$ & \color{red}{$K_{2,3,5}$} & \\ 
\hline
42 & $(D_{21};C_{21},C_2\colon5)_{\text{c}}$ & \color{red}{$K_{2,3,7}$} & \\ 
\hline
48 & $(\mathrm{CSU}_2(\mathbb{F}_3);\mathrm{SL}_2(\mathbb{F}_3),C_2\colon28)_{\text{d}}$ & $9_{35},9_{46}$ & $9$ \\
& $(\mathrm{GL}_2(\mathbb{F}_3);\mathrm{SL}_2(\mathbb{F}_3),C_2\colon29)_{\text{e}}$ & $9_{35},9_{46}$ & $9$ \\
& $(A_4\rtimes C_4;A_4,C_4\colon30)_{\text{a}}$ & $9_{35},9_{46}$ & $9$ \\ 
\hline
60 & $(\mathrm{Dic}_{15};C_{15},C_4\colon3)_{\text{b}}$ & $K_{2,3,5}$ & \\
 & $(A_5;A_5,\{e\}\colon5)_{\text{f}}$ & $9_{35},9_{46}$, \color{red}{$10_{67},10_{120},10_{146},6_1^{(3,5)}$} & $9$ \\
 & $(C_3\rtimes F_5;C_{15},C_4\colon7)_{\text{g}}$ & \color{red}{$6_1^{(3,5)}$} & \\
\hline
66 & $(D_{33};C_{33},C_2\colon3)_{\text{h}}$ & \color{red}{$K_{2,3,11}$} & \\
\hline
70 & $(D_{35};C_{35},C_2\colon3)_{\text{i}}$ & \color{red}{$K_{2,5,7}$} & \\ 
\hline
72 & $(C_3.S_4;C_3.A_4,C_2\colon15)_{\text{j}}$ & $9_{35}$ & $9$ \\
 & $(C_3\times S_4;A_4,C_6\colon42)_{\text{a}}$ & $9_{35},9_{46}$ & $9$ \\
 & $(C_3\rtimes S_4;C_3\times A_4,C_2\colon43)_{\text{k}}$ & $9_{35},9_{46}$ & $9$ \\
 & $(S_3\times A_4;C_2\times C_6,C_6\colon44)_{\text{l}}$ & & \\
\hline
78 & $(D_{39};C_{39},C_2\colon5)_{\text{m}}$ & \color{red}{$K_{2,3,13}$} & \\ 
\hline
84 & $(\mathrm{Dic}_{21};C_{21},C_4\colon5)_{\text{c}}$ & $K_{2,3,7}$ & \\
 & $(C_7\rtimes A_4;C_2\times C_{14},C_3\colon11)_{\text{n}}$ & \color{red}{$8_{11}^{(2,7)}$} & \\
\hline
90 & $(D_{45};C_{45},C_2\colon3)_{\text{o}}$ & & \\ 
 & $(C_3\times D_{15};C_{15},C_6\colon7)_{\text{b}}$ & $K_{2,3,5}$ & \\
 & $(C_3\rtimes D_{15};C_3\times C_{15},C_2\colon9)_{\text{p}}$ & & \\
\hline
96 & $(C_4^2\rtimes S_3; C_4^2\rtimes C_3,C_2\colon64)_{\text{q}}$ & $9_{35},9_{46}$ & $9$ \\
& $(A_4\rtimes C_8;A_4,C_8\colon65)_{\text{a}}$ & $9_{35},9_{46}$ & $9$ \\
& $(Q_8\rtimes\mathrm{Dic}_3;\mathrm{SL}_2(\mathbb{F}_3),C_4\colon66)_{\text{de}}$ & $9_{35},9_{46}$ & $9$ \\
& $(\mathrm{U}_2(\mathbb{F}_3);\mathrm{SL}_2(\mathbb{F}_3),C_4\colon67)_{\text{r}}$ & $9_{35},9_{46}$ & $9$ \\
& $(C_2^2\rtimes S_4;C_2^2\rtimes A_4,C_2\colon227)_{\text{s}}$ & $9_{35},9_{46},10_{120},\color{red}{10_{166}}$ & $9$ \\
\hline
102 & $(D_{51};C_{51},C_2\colon3)$ & \color{red}{$K_{2,3,17}$} & \\ 
\hline
110 & $(D_{55};C_{55},C_2\colon5)$ & \color{red}{$K_{2,5,11}$} & \\ 
\hline
114 & $(D_{57};C_{57},C_2\colon5)$ & \color{red}{$K_{2,3,19}$} & \\ 
\hline
120 & $(C_{15}\rtimes_3C_8;C_{15},C_8\colon3)_{\text{b}}$ & $K_{2,3,5}$ & \\
 & $(\mathrm{SL}_2(\mathbb{F}_5);\mathrm{SL}_2(\mathbb{F}_5),\{e\}\colon5)$ & $9_{35},9_{46},10_{67},10_{120},10_{146}$ & $9$ \\
 & $(C_{15}\rtimes C_8;C_{15},C_8\colon7)_{\text{g}}$ & $6_1^{(3,5)}$ & \\
 & $(S_5;A_5,C_2\colon34)$ & ${\color{red}{8_{15},9_{25}}},9_{35},\color{red}{9_{39},9_{41},9_{49},10_{58}}$ & $8$ \\
 & $(C_2\times A_5;A_5,C_2\colon35)_{\text{f}}$ & $9_{35},9_{46},10_{67},10_{120},10_{146},6_1^{(3,5)}$ & $9$ \\
 & $(C_5\times S_4;A_4,C_{10}\colon37)_{\text{a}}$ & $9_{35},9_{46}$ & $9$ \\
 & $(C_5\rtimes S_4;C_5\times A_4,C_2\colon38)$ & & \\
 & $(D_5\times A_4;C_2\times C_{10},C_6\colon39)$ & & \\
\hline
126 & $(D_{63};C_{63},C_2\colon5)$ & & \\
 & $(S_3\times C_7\rtimes C_3;C_{21},C_6\colon8)$ & & \\
 & $(C_3\rtimes F_7;C_{21},C_6\colon9)$ & & \\
 & $(C_3\times D_{21};C_{21},C_6\colon13)_{\text{c}}$ & $K_{2,3,7}$ & \\
 & $(C_3\rtimes D_{21};C_3\times C_{21},C_2\colon15)$ & & \\
\hline
\end{tabular}

\vspace{2mm}

\caption{TAV groups of order less than $130$.}

\end{center}
\end{table}

\begin{table}\label{table:3}
\begin{center}
\begin{tabular}{|c|l|l|c|}
\hline
$|G|$ & $(G;G',G/G'\colon\#\, \text{GroupNames})$ & $K$ & $c$ \\
\hline
130 & $(D_{65};C_{65},C_2\colon3)$ & \color{red}{$K_{2,5,13}$} & \\
\hline
132 & $(\mathrm{Dic}_{33};C_{33},C_4\colon3)_{\text{h}}$ & $K_{2,3,11}$ & \\
\hline
138 & $(D_{69};C_{69},C_2\colon3)$ & \color{red}{$K_{2,3,23}$} & \\
\hline
140 & $(\mathrm{Dic}_{35};C_{35},C_4\colon3)_{\text{i}}$ & $K_{2,5,7}$ & \\
 &  $(C_7\rtimes F_5;C_{35},C_4\colon6)$ & & \\
\hline
144 & $(Q_8.D_9;Q_8\rtimes C_9,C_2\colon31)$ & $9_{35}$ & $9$ \\
& $(Q_8\rtimes D_9;Q_8\rtimes C_9,C_2\colon32)$ & $9_{35}$ & $9$ \\
& $(C_6.S_4;C_3.A_4,C_4\colon33)_{\text{j}}$ & $9_{35}$ & $9$\\
& $(C_3\times \mathrm{CSU}_2(\mathbb{F}_3);\mathrm{SL}_2(\mathbb{F}_3),C_6\colon121)_{\text{d}}$ & $9_{35},9_{46}$ & $9$ \\
& $(C_3\times \mathrm{GL}_2(\mathbb{F}_3);\mathrm{SL}_2(\mathbb{F}_3),C_6\colon122)_{\text{e}}$ & $9_{35},9_{46}$ & $9$ \\
& $(C_3\times A_4\rtimes C_4;A_4,C_{12}\colon123)_{\text{a}}$ & $9_{35},9_{46}$ & $9$ \\
& $({C_6.}_{5}S_4;C_3\times \mathrm{SL}_2(\mathbb{F}_3),C_2\colon124)$ & $9_{35},9_{46}$ & $9$ \\
& $({C_6.}_{6}S_4;C_3\times \mathrm{SL}_2(\mathbb{F}_3),C_2\colon125)$ & $9_{35},9_{46}$ & $9$ \\
& $({C_6.}_{7}S_4;C_3\times A_4,C_4\colon126)_{\text{k}}$ & $9_{35},9_{46}$ & $9$ \\
& $(\mathrm{Dic}_3.A_4;C_3\times Q_8,C_6\colon127)$ & & \\
& $(S_3\times \mathrm{SL}_2(\mathbb{F}_3);C_3\times Q_8,C_6\colon128)$ & & \\
& $(\mathrm{Dic}_3\times A_4;C_2\times C_6,C_{12}\colon129)_{\text{l}}$ & & \\
\hline
150 & $(D_{75};C_{75},C_2\colon3)$ & & \\
 &  $(C_5^2\rtimes S_3;C_5^2\rtimes C_3,C_2\colon5)$ & & \\
 & $(C_5\times D_{15};C_{15},C_{10}\colon11)_{\text{b}}$ & $K_{2,3,5}$ & \\
 & $(C_5\rtimes D_{15};C_5\times C_{15},C_2\colon12)$ & & \\
\hline
154 & $(D_{77};C_{77},C_2\colon3)$ & \color{red}{$K_{2,7,11}$} & \\
\hline
156 & $(\mathrm{Dic}_{39};C_{39},C_4\colon5)_{\text{m}}$ & $K_{2,3,13}$ & \\
 & $(C_{39}\rtimes C_4;C_{39},C_4\colon10)$ & & \\
 & $(C_{13}\rtimes A_{4};C_2\times C_{26},C_3\colon14)$ & & \\
\hline
160 & $(C_2^4\rtimes D_5;C_2^4\rtimes C_5,C_2\colon234)$ & & \\
\hline
168 & $(C_{21}\rtimes C_8;C_{21},C_8\colon5)_{\text{c}}$  & $K_{2,3,7}$ & \\
 & $(C_{14}.A_4;C_7\times Q_8,C_3\colon23)$ & & \\
 & $(\mathrm{GL}_3(\mathbb{F}_2);\mathrm{GL}_3(\mathbb{F}_2),\{e\}\colon42)$ & $\color{red}{7_4,8_3}$ & $7$ \\
& $(\mathrm{A\Gamma L}_1(\mathbb{F}_8);F_8,C_3\colon43)$ & & \\
& $(C_7\times S_4;A_4,C_{14}\colon45)_{\text{a}}$ & $9_{35},9_{46}$ & $9$ \\
& $(C_7\rtimes S_4;C_7\times A_4,C_2\colon46)$ & & \\
& $(A_4\times D_7;C_2\times C_{14},C_6\colon48)$ & & \\
& $(D_7\rtimes A_4;C_2\times C_{14},C_6\colon49)$ & & \\
& $(C_2\times C_7\rtimes A_4;C_2\times C_{14},C_6\colon53)_{\text{n}}$ & $8_{11}^{(2,7)}$ & \\
\hline
\end{tabular}

\vspace{2mm}

\caption{TAV groups of order $130$ to $169$.}

\end{center}
\end{table}

\begin{table}\label{table:3}
\begin{center}
\begin{tabular}{|c|l|l|c|}
\hline
$|G|$ & $(G;G',G/G'\colon\#\, \text{GroupNames})$ & $K$ & $c$ \\
\hline
170 & $(D_{85};C_{85},C_2\colon3)$ & \color{red}{$K_{2,5,17}$} & \\
\hline
174 & $(D_{87};C_{87},C_2\colon3)$ & \color{red}{$K_{2,3,29}$} & \\
\hline
180 & $(\mathrm{Dic}_{45};C_{45},C_4\colon3)_{\text{o}}$ & & \\
& $(C_9\rtimes F_5;C_{45},C_4\colon6)$ & & \\
& $(C_3\times \mathrm{Dic}_{15};C_{15},C_{12}\colon15)_{\text{b}}$ & $K_{2,3,5}$ & \\
& $(C_3\rtimes \mathrm{Dic}_{15};C_3\times C_{15},C_4\colon17)_{\text{p}}$ & & \\
& $(C_3\times A_5;A_5,C_3\colon19)_{\text{f}}$ & $9_{35},9_{46},10_{67},10_{120},10_{146},6_1^{(3,5)}$ & $9$ \\
& $(C_3\times C_3\rtimes F_5;C_{15},C_{12}\colon21)_{\text{g}}$ & $6_1^{(3,5)}$ &  \\
& $(C_3^2\rtimes_3 F_5;C_3\times C_{15},C_4\colon22)$ & & \\
& $(C_3^2\rtimes\mathrm{Dic}_5;C_3\times C_{15},C_4\colon24)$ & & \\
& $(C_3^2\rtimes F_5;C_3\times C_{15},C_4\colon25)$ & & \\
\hline
182 & $(D_{91};C_{91},C_2\colon3)$ & \color{red}{$K_{2,7,13}$} & \\
\hline
186 & $(D_{93};C_{93},C_2\colon5)$ & \color{red}{$K_{2,3,31}$} & \\
\hline
190 & $(D_{95};C_{95},C_2\colon3)$ & \color{red}{$K_{2,5,19}$} & \\
\hline
192 & $({C_2^3.}_7S_4;{C_2^3.}_3 A_4,C_2\colon180)$ & $9_{35},9_{46}$ & $9$ \\
& $({C_2^3.}_8S_4;{C_2^3.}_3 A_4,C_2\colon181)$ & $9_{35},9_{46}$ & $9$ \\
& $({C_2^3.}_9S_4;{C_4^2}\rtimes C_3,C_4\colon182)_{\text{q}}$ & $9_{35},9_{46}$ & $9$ \\
& $(C_2.\mathrm{U}_2(\mathbb{F}_3);\mathrm{SL}_2(\mathbb{F}_3),C_8\colon183)_{\text{der}}$ & $9_{35},9_{46}$ & $9$ \\
& $(C_2^4\rtimes \mathrm{Dic}_3;C_2^2\rtimes A_4,C_4\colon184)$ & $9_{35},9_{46},10_{120},10_{166}$ & $9$ \\
& $(C_4^2\rtimes\mathrm{Dic}_3;C_4^2\rtimes C_3,C_4\colon185)$ & $9_{35},9_{46}$ & $9$ \\
& $(A_4\rtimes C_{16};A_4,C_{16}\colon186)_{\text{a}}$ & $9_{35},9_{46}$ & $9$ \\
&$({C_8.}_7S_4;\mathrm{SL}_2(\mathbb{F}_3),C_8\colon187)$ & $9_{35},9_{46}$ & $9$ \\
& $({Q_8.}_1S_4;Q_8\rtimes A_4,C_2\colon1489)$ & $9_{35},9_{46},10_{120},10_{166}$  & $9$ \\
& $(Q_8\rtimes S_4;Q_8\rtimes A_4,C_2\colon1490)$ & $9_{35},9_{46},10_{120},10_{166}$ & $9$ \\
& $(C_2^3.S_4;C_2^3\rtimes A_4,C_2\colon1491)$ & $9_{35},9_{46},10_{120},10_{166}$ & $9$ \\
& $(Q_8.S_4;C_2^3\rtimes A_4,C_2\colon1492)$ & $9_{35},9_{46},10_{120},10_{166}$  & $9$ \\
& $(C_2^3\rtimes S_4;C_2^3\rtimes A_4,C_2\colon1493)$ & $9_{35},9_{46},10_{120},10_{166}$ & $9$ \\
& $(Q_8\rtimes_2S_4;C_2^3\rtimes A_4,C_2\colon1494)$ & $9_{35},9_{46},10_{120},10_{166}$ & $9$ \\
& $(C_2^4\rtimes_4\mathrm{Dic}_3;C_2^2\rtimes A_4,C_4\colon1495)_{\text{s}}$ & $9_{35},9_{46},10_{120},10_{166}$ & $9$ \\
\hline
198 & $(D_{99};C_{99},C_2\colon3)$ & & \\
& $(C_3\times D_{33};C_{33},C_6\colon7)_{\text{h}}$ & $K_{2,3,11}$ & \\
& $(C_3\rtimes D_{33};C_3\times C_{33},C_2\colon9)$ & & \\
\hline
\end{tabular}

\vspace{2mm}

\caption{TAV groups of order $170$ to $200$.}

\end{center}
\end{table}

\subsection*{Acknowledgments}
The authors are supported in part by 
JSPS KAKENHI Grant Numbers JP23K20799, 
JP25K07012, and JP25K17252. 

\bibliographystyle{amsplain}

\end{document}